\documentclass{amsart}
\usepackage{amssymb,amsmath}

\usepackage{verbatim}

\usepackage{xcolor}
\usepackage{pb-diagram}
\usepackage{pstricks}
\usepackage{pdftricks}
\usepackage{pst-node}

\newcommand{\al}{\alpha}

\newcommand{\BP}{\mathcal{B}}
\newcommand{\BPQ}{\mathcal{QB}}
\newcommand{\BPQR}{\overleftarrow{\mathcal{QB}}}
\newcommand{\bbeta}{\overline{\beta}}
\newcommand{\be}{\beta}
\newcommand{\bo}{\mathbf{1}}
\newcommand{\dir}[1]{\mathrm{dir}(#1)}
\newcommand{\ep}{\mathrm{end}}

\newcommand{\HH}{\mathcal{H}}
\newcommand{\hJ}{\hat{J}}

\newcommand{\id}{\mathrm{id}}
\newcommand{\la}{\lambda}
\newcommand{\ord}{\mathrm{ord}}
\newcommand{\pair}[2]{\langle #1\,,\,#2\rangle}
\newcommand{\QB}{\mathrm{QB}}
\newcommand{\qwt}{\mathrm{qwt}}
\newcommand{\R}{\mathbb{R}}
\newcommand{\tE}{\tilde{E}}
\newcommand{\wa}{\vec{w}}
\newcommand{\wt}{\mathrm{wt}}
\newcommand{\Z}{\mathbb{Z}}

\newcommand{\Aut}{\mathrm{Aut}}
\newcommand{\Daut}{\Pi}
\newcommand{\ds}{*}
\newcommand{\fixit}[1]{\texttt{{\color{red} *** #1 *** }}}

\newcommand{\Hom}{\mathrm{Hom}}
\newcommand{\Inv}{\mathrm{Inv}}
\newcommand{\K}{\mathbb{K}}
\newcommand{\Pol}{\mathcal{V}}
\newcommand{\Q}{\mathbb{Q}}
\newcommand{\rd}{\mathrm{red}}
\newcommand{\rootco}{\eta}
\newcommand{\spc}{\mathrm{sp}}
\newcommand{\tX}{\tilde{X}}
\newcommand{\tY}{\tilde{Y}}
\newcommand{\Xv}{X^\ds}

\renewcommand{\tilde}{\widetilde}
\renewcommand{\comment}[1]{}
\renewcommand{\H}{\mathcal{H}}
\newcommand{\Ind}{\mathrm{Ind}}

\newcommand{\de}{\delta}

\newtheorem{lem}{Lemma}
\newtheorem{thm}[lem]{Theorem}
\newtheorem{prop}[lem]{Proposition}
\newtheorem{cor}[lem]{Corollary}

\theoremstyle{remark}

\newtheorem{rem}[lem]{Remark}
\newtheorem{ex}[lem]{Example}

\numberwithin{equation}{section}
\numberwithin{lem}{section}

\title{Specializations of nonsymmetric Macdonald-Koornwinder polynomials}

\author{D. Orr}
\address{
Department of Mathematics, MC 0123,
460 McBryde Hall, Virginia Tech,
225 Stanger St., Blacksburg, VA 24061 USA}
\email{dorr@vt.edu}

\author{M. Shimozono}
\address{Department of Mathematics, MC 0123,
460 McBryde Hall, Virginia Tech,
225 Stanger St., Blacksburg, VA 24061 USA}
\email{mshimo@math.vt.edu}

\begin{document}

\begin{abstract} This work records the details of the Ram-Yip formula for nonsymmetric Macdonald-Koornwinder polynomials for the double affine Hecke algebras of not-necessarily-reduced affine root systems. It is shown that the $t\to0$ equal-parameter specialization of nonsymmetric Macdonald polynomials admits an explicit combinatorial formula in terms of quantum alcove paths, generalizing the formula of Lenart in the untwisted case. In particular our formula yields a definition of quantum Bruhat graph for all affine root systems. For mixed type the proof requires the Ram-Yip formula for the nonsymmetric Koornwinder polynomials. A quantum alcove path formula is also given at $t\to\infty$. As a consequence we establish the positivity of the coefficients of nonsymmetric Macdonald polynomials under this limit, as conjectured by Cherednik and the first author. Finally, an explicit formula is given at $q\to\infty$, which yields the $p$-adic Iwahori-Whittaker functions of Brubaker, Bump, and Licata.
\end{abstract}

\maketitle


\section{Introduction}

\subsection{Ram-Yip formula}
The nonsymmetric Macdonald-Koornwinder polynomials $E_\la(X;q;t_\bullet)$ 
\cite{Che:1995} \cite{Mac:1996} \cite{O} \cite{Sah}
form a remarkable basis of the polynomial module of the double affine Hecke algebra (DAHA). The Ram-Yip formula \cite{RY} gives an explicit expression for $E_\lambda(X;q;t_\bullet)$ that is particularly suitable for combinatorial study. Although in \cite{RY} the formula was stated for the equal-parameter DAHAs of reduced affine root systems, it was intended to work for general DAHAs, and indeed it does. The first goal of this paper is to record the details of the generalization of the Ram-Yip formula to the unequal parameter nonreduced case, which includes the Koornwinder polynomials. A similar generalization is given for the symmetric Macdonald-Koornwinder polynomials $P_\lambda(X;q;t_\bullet)$.

We study the behavior of the Ram-Yip formula at various specializations. At $t\to0$, $E_\lambda(X;q;0)$ and $P_\lambda(X;q;0)$ are related to affine Demazure characters \cite{I:2003} \cite{San}, Kirillov-Reshetikhin characters \cite{FL} \cite{LNSSS}, characters of Weyl modules of current algebras \cite{ChaFK} \cite{FL}, and projected level zero Lakshmibai-Seshadri paths \cite{LNSSS}. We give a formula for $E_\lambda(X;q;0)$ in terms of quantum alcove paths, generalizing to arbitrary affine root systems the formulas of \cite{Len} \cite{LenSc} for $P_\la(X;q;0)$ in the untwisted case.

At $t\to\infty$ the polynomials $E_\lambda(X;q;\infty)$ arise in the study of the finite-difference Toda lattice and are conjecturally related to the PBW filtration of affine Demazure modules \cite{CF} \cite{CO}. We give an alcove path formula for $E_\lambda(X;q;\infty)$ in terms of reverse paths in the quantum Bruhat graph. This establishes the positivity of the coefficients of $E_\lambda(X;q;\infty)$ for all affine root systems, as conjectured in \cite{CO} in the dual untwisted setting.

At $q\to\infty$, the $E_\lambda(X;\infty;t)$ for untwisted affine root systems are the $p$-adic Iwahori-Whittaker functions of Brubaker, Bump, and Licata \cite{BBL}.\footnote{In \cite{BBL}, the authors use $E_\lambda(X;q^{-1};t^{-1})$ and correspondingly send $q\to0$;
\cite{I:2008} also uses this convention for $q$ and $t$.} We give an alcove path formula for the specialization $q\to\infty$ valid for all affine root systems. We also consider the limit $q\to0$, which is the case that inspired the Ram-Yip formula: $P_\lambda(X;0;t)$ is the spherical function known as the Hall-Littlewood polynomial, and the Ram-Yip formula degenerates to Schwer's formula \cite{Sc}. The specializations $E_\la(X;0;t)$ and $E_\la(X;\infty;t)$ also have interpretations as standard and dual standard bases in Kazhdan-Lusztig theory \cite{I:2008}.

We would be remiss not to mention the formula of Haglund, Haiman, and Loehr \cite{HHL} for $E_\lambda(X;q;t)$ in type $A$, and the many deep related works for the modified Macdonald polynomials.

Thanks to Bogdan Ion, Arun Ram, and Siddartha Sahi, for patient explanations about the double affine Hecke algebra.
Thanks to Anne Schilling and Nicolas Thi\'ery and for implementing nonsymmetric Macdonald polynomials in \texttt{sage} \cite{Sage-Combinat}; we used their program extensively.
Thanks to ICERM, which provided the venue for the above activities. Thanks to Christian Lenart, Satoshi Naito, Daisuke Sagaki, and Anne Schilling for related collaborations.
Thanks to Ivan Cherednik for helpful discussions.
The second author thanks the NSF for the support from grant
NSF DMS-1200804.

\section{DAHA}
Our exposition of the DAHA follows Haiman \cite{Hai} but with important notational differences. We consistently name an object according to the relations it satisfies or by its behavior, regardless of the origins of its ingredients.

\subsection{Cartan data}
A Cartan datum is a pair $(I,A)$ where $I$ is a finite set
and $A=(a_{ij}\mid i,j\in I)$ is a generalized Cartan matrix, that is, an integer matrix with $a_{ii}=2$ for all $i\in I$, and for all $i,j\in I$ with $i\ne j$, $a_{ij}\le 0$
with $a_{ij}<0$ if and only if $a_{ji}<0$. This defines the Weyl group $W=W(I,A)$, the Coxeter group generated by elements $s_i$ for $i\in I$ satisfying $s_i^2=1$ and
for $i\ne j$, $(s_i s_j)$ has order $2,3,4,6,\infty$
according as $a_{ij} a_{ji}$ has value $0,1,2,3,\ge 4$.

\begin{ex}
\label{X:D2Cartan} The Cartan datum for the dual untwisted affine root system $D_{2+1}^{(2)}$ has affine Dynkin node set $I=\{0,1,2\}$ with Cartan matrix
\begin{align*}
\begin{pmatrix}
2 & -2 & 0 \\
-1 & 2 & -1 \\
0 & -2 & 2
\end{pmatrix}.
\end{align*}
The affine Weyl group $W_a(D_{2+1}^{(2)})$ is generated by involutions $s_0,s_1,s_2$ with braid relations $(s_0s_1)^4=(s_1s_2)^4=(s_0s_2)^2=\id$.
\end{ex}

\subsection{Root data}
Given a Cartan datum $(I,A)$, a root datum is a triple
$(X,\{\alpha_i\mid i\in I\}, \{\alpha_i^\vee\mid i\in I\})$
where $X$ is a lattice (free $\Z$-module) containing elements
$\alpha_i$ (simple roots) for $i\in I$, and $\alpha_i^\vee$ (simple coroots) are elements
in $\Xv=\Hom_\Z(X,\Z)$ such that
\begin{align}\label{E:Cartan}
  \pair{\alpha_i^\vee}{\alpha_j} = a_{ij} \qquad\text{for $i,j\in I$}
\end{align}
where $\pair{\cdot}{\cdot}$ is the evaluation pairing.

The Weyl group $W=W(I,A)$ acts on $X$ and $\Xv$ by
$s_i(\la)=\la-\pair{\alpha_i^\vee}{\la}\alpha_i$
and $s_i(\mu)=\mu-\pair{\mu}{\alpha_i}\alpha_i^\vee$
for $i\in I$, $\la\in X$, and $\mu\in \Xv$. The pairing
$\pair{\cdot}{\cdot}$ is $W$-invariant.

Denote by $R=R(X)= \bigcup_{i\in I} W(\alpha_i)$ the set of real roots. For $\alpha = w(\alpha_i)\in R(X)$, denote by
$\alpha^\vee=w(\alpha_i^\vee)\in \Xv$ its associated coroot and by $s_\alpha=w s_i w^{-1}\in W$ its associated reflection. We have $R = R_+ \bigsqcup -R_+$
where $R_+ = R \cap (\bigoplus_{i\in I} \Z_{\ge0} \alpha_i)$ is the set of positive real roots.

\begin{ex} \label{X:D2root} Consider the free $\Z$-module
$\tX=\Z\delta\oplus \Z e_1 \oplus \Z e_2$.
Define $\alpha_0=\delta-e_1$,  $\alpha_1=e_1-e_2$, $\alpha_2=e_2$. Let $\alpha_0^\vee,\alpha_1^\vee,\alpha_2^\vee\in \tX^\ds = \Hom_\Z(\tX,\Z)$ be defined by 
$\alpha_0^\vee=-2e_1^*$,
$\alpha_1^\vee = e_1^*-e_2^*$ and $\alpha_2^\vee=2e_2^*$. This yields a root datum for the Cartan datum of Example \ref{X:D2Cartan}. The affine Weyl group $W_a(\tX)$ of Example \ref{X:D2Cartan} acts on the set $R$ of real roots by the orbits $W_a(\tX)(\alpha_0) = (2\Z+1)\delta + W_0 \alpha_2$, $W_a(\tX)(\alpha_1) = 2\Z\delta + W_0 \alpha_1$, 
$W_a(\tX)(\alpha_2) = 2\Z\delta + W_0 \alpha_2$.
\end{ex}

\subsection{Double affine data}
\label{SS:doubleaffinedata}
A \textit{double affine datum} consists of two root data
\begin{align*}
(X,\{\alpha^X_i\},\{\alpha_i^{\vee X} \})\qquad
(Y,\{\alpha^Y_i\},\{\alpha_i^{\vee Y} \})
\end{align*}
of irreducible reduced finite type together with
an isomorphism of finite Weyl groups $W_0(X)\cong W_0(Y)$ 
via $s_i^X \mapsto s_i^Y$ for all $i\in I_0$ where $I_0$ is the finite Dynkin node set. The isomorphism $W_0(X)\cong W_0(Y)$ implies that either $X$ and $Y$ have the same Cartan data or dual Cartan data, the latter meaning that the Cartan matrices are transposes of each other.

We require that $X$ and $Y$ have rank $|I_0|$, so that
\begin{align}
\label{E:Xbetween}
Q^X \subset &\ X \subset P^X \\
\label{E:Ybetween}
Q^Y \subset &\ Y \subset P^Y
\end{align}
where $Q$ and $P$ denote the root and weight lattices respectively.

Let $\omega_i^X\in P^X$ for $i \in I_0$ be the fundamental weights, which are defined by
$\pair{\alpha_j^{\vee X}}{\omega_i^X}=\delta_{ij}$ for all $j\in I_0$. Let $P^X_+=\bigoplus_{i\in I_0} \Z_{\ge0} \omega_i^X$ be the set of dominant weights of $X$ and $X_+ = P^X_+ \cap X$. Let $P^X_-=-P^X_+$ and $X_-=-X_+$ be the corresponding sets of antidominant weights. We make similar definitions for $Y$.

\subsection{Pair of affine root data}
\label{SS:affinerootdata}
For a double affine datum $(X,Y)$, let
$\tX_\rd$ and $\tY_\rd$ be the associated pair of irreducible reduced affine root data, defined
as follows. Let $I=I_0 \cup \{0\}$ be the affine
Dynkin node set, with distinguished affine node $0\in I$.

Suppose $X$ and $Y$ have dual Cartan data.
Let $\tX$ (resp. $\tY$) stand for the affine Cartan data obtained from $X$ (resp. $Y$) by untwisted affinization. In this case we say that $(X,Y)$ is of untwisted type.

Suppose $X$ and $Y$ share a common Cartan datum.
Then the common affine Cartan data of $\tX$ and $\tY$ is  obtained from $X$ by taking the dual, then the untwisted affinization, and then the affine dual. The result is a Cartan datum of dual untwisted
(the dual of an untwisted) affine type. Hence we say that such $(X,Y)$ has dual untwisted type.

\begin{ex} \label{X:dad}
\begin{enumerate}
\item 
Let $X$ and $Y$ be of type $A_n$. Then $\tX$ and $\tY$ are of type $A_n^{(1)}$.
\item 
Let $X$ and $Y$ be of type $B_n$. Then $\tX$ and $\tY$ are of type $D_{n+1}^{(2)}$.
\item
Let $X$ and $Y$ be of types $B_n$ and $C_n$ respectively. Then $\tX$ and $\tY$ are of types $B_n^{(1)}$ and $C_n^{(1)}$ respectively.
\end{enumerate}
\end{ex}

\begin{rem} \label{R:nomixed} The mixed affine root systems
$A_{2n}^{(2)}$ (with $\alpha_0$ extra short) and
its dual $A_{2n}^{(2)\dagger}$ do not occur as either
$\tX$ or $\tY$. Their Macdonald polynomials may be obtained by specializations of Koornwinder polynomials, which come from the double affine datum $(Q(B_n),Q(B_n))$. See
Remark \ref{R:mixedtype} and \S \ref{SS:A2}, \ref{SS:A2dag}.
\end{rem}

Let $\delta^X = \alpha_0^X + \theta^X$ be the affine null root, where $\theta^X\in X$. In the untwisted case, $\theta^X$ is the high root and in the dual untwisted case, it is the dominant short root, that is, the root associated with the highest coroot. Define the lattice $\tX= X \oplus \Z\delta^X$. For $i\in I_0$, let $\alpha_i^\vee \in \tX^\ds:=\Hom_\Z(\tX,\Z)$ be defined by the extension of $\alpha_i^\vee\in \Xv$ to $\tX$ by $\pair{\alpha_i^\vee}{\delta^X}=0$.
Define $\alpha_0^\vee$ by $\pair{\alpha_0^\vee}{\alpha_j}=a_{0j}$ (the Cartan matrix entry) for $j\in I_0$ and $\pair{\alpha_0^\vee}{\delta^X}=0$.
One may show that this implies $\pair{\alpha_0^\vee}{\alpha_0}=2$, thereby defining an affine root datum 
$(\tX,\{\alpha_0,\dotsc,\alpha_n\},\{\alpha_0^\vee,\dotsc,\alpha_n^\vee\})$. Similar definitions apply for $\tY$.

Let $\vartheta^X\in X$ be the dominant short root. 
Thus, in the untwisted case $\vartheta^X$
is the root associated to $\theta^Y$ (viewed as a coroot of $X$), and in the dual untwisted case
$\vartheta^X=\theta^X$.

For $\mu=\be+a\de^Y\in\tY$ with $\beta\in Y$, let $\overline{\mu}=\beta$ be the projection to the classical weight lattice.
Define the function $\deg: \bigoplus_{i\in I^Y} \Z \alpha_i^Y \to \Z$ by
\begin{align}\label{E:deg}
  \deg(a\delta^Y + \beta) = a\qquad\text{where $\beta\in Y$}.
\end{align}

\subsection{Reduced affine root systems}
Let $W_a=W_a(\tX)$ denote the (nonextended) affine Weyl group, which is generated by reflections $s_i^X$ for $i\in I^X$, and denote by $R_\rd(\tX) = \bigcup_{i\in I} W_a(\alpha_i)$ the set of real roots. The subscript $\rd$ is used to distinguish this set of roots (the real roots in the reduced affine root system) from a possibly larger set of roots $R(\tX)$ which shall be defined later.
Let $R_\rd(\tY)$ be defined similarly.

\subsection{Pairing}\label{SS:pairing}
For a lattice $L$ let $L_\Q = \Q\otimes_\Z L$.
The goal is to define a $W_0$-invariant pairing $(\cdot,\cdot): Y_\Q \times X_\Q \to \Q$.
By \eqref{E:Xbetween} and \eqref{E:Ybetween}
$X_\Q = Q^X_\Q$ and $Y_\Q = Q^Y_\Q$, so  it suffices to define $(,)$ on $Q^Y \times Q^X$. 

For any $\alpha\in R_+(Y)$, let $\gamma_\alpha\in \Z_{>0}$ be minimum such that $\gamma_\alpha \delta^Y -\alpha \in R_\rd(\tY)$. For $i\in I_0$ define $\gamma_i=\gamma_{\alpha_i}$. 
 Define
\begin{align}
\label{E:rootco}
  \rootco: Q^Y &\to Q^{\vee X} \\
  \rootco(\alpha_i^Y) &= \gamma_i \alpha_i^{X\vee}\qquad\text {for $i\in I_0$}
\end{align}
where $Q^\vee$ is the coroot lattice. The map $\eta$ is a $W_0$-equivariant embedding. Define
\begin{align}\label{E:pairing}
  (\mu,\la) = \pair{\rootco(\mu)}{\la}\qquad\text{for $\mu\in Q^Y$, $\la\in Q^X$}
\end{align} 
where $\pair{\cdot}{\cdot}$ is the evaluation pairing on $Q^{\vee X}\times Q^X$. In the dual untwisted case, $(,)$ is the $W_0$-invariant pairing on $Q^X \times Q^X$
under which short roots have square length $2$. 

\begin{rem}\label{R:DUpair} 
\begin{enumerate}
\item
If $(X,Y)$ is untwisted, $\gamma_i=1$ for all $i\in I_0$ and the map $\eta$ is a canonical $W_0$-equivariant isomorphism.

\item If $(X,Y)$ is dual untwisted, $(\gamma_i\mid i\in I_0)$ is the Cartan symmetrizer, the unique $I_0$-tuple of positive integers such that $\gamma_i=1$ for $\alpha_i^Y \in Y$ short and the matrix $(\gamma_i a_{ij}\mid i,j\in I_0)$ is symmetric, where $(a_{ij}\mid i,j\in I_0)$ is the Cartan matrix for $Y$. Moreover, for $\alpha_i^Y$ long,
$\gamma_i=r$, where $r\in \{2,3\}$ is the multiplicity of the multiple bond in the Dynkin diagram of $Y$, or equivalently, the twist of the affine root system $\tilde{Y}$, e. g., the $2$ in $D_{n+1}^{(2)}$ or the 3 in $D_4^{(3)}$ in the notation of \cite{Kac}.
By symmetry and $W_0$-invariance, for any root $\beta\in R(X)\subset Q^X=Q^Y$ and weight $\mu\in P^Y=P^X$ we have
\begin{align}\label{E:pair}
  (\mu,\beta) = (\beta,\mu) = 
  \begin{cases} 
  \pair{\beta^\vee}{\mu} & \text{if $\beta$ is short} \\
  r \pair{\beta^\vee}{\mu} & \text{if $\beta$ is long.}
  \end{cases}
\end{align}
\end{enumerate}
\end{rem}

Fix $m\in \Z_{>0}$ such that $(Y,X)\subset (1/m)\Z$. By Remark \ref{R:DUpair}, we have $(Y,Q^X)\in\Z$.

\subsection{Affine Weyl groups}
For the affine root datum $\tX$, denote by
$W_a(\tX)$ and $W_e(\tX)$ its affine and extended affine Weyl groups respectively. 

If $\tX$ is untwisted, $W_a(\tX)=W_a\cong W_0 \ltimes Q^\vee$ and $W_e(\tX)=W_e\cong W_0 \ltimes P^\vee$ where $Q^\vee$ and $P^\vee$ are the coroot and coweight lattices of the classical subsystem $X$.
For $\mu \in P^\vee$, we write $t_\mu$ for the corresponding element of $W_e(\tX)$ given by this decomposition. One has
\begin{align}\label{E:tmux}
t_\mu(x)=x-\pair{\mu}{x}\de\qquad\text{for $\mu\in P^\vee$, $x\in Q$}
\end{align}
and $t_\mu(\de)=\de$.

There is an isomorphism $W_e \cong \Daut \ltimes W_a$ where
$\Daut$ is the subgroup of length zero elements in $W_e$.
The group $\Daut$ may be realized as a subgroup of the group $\Aut=\Aut(\tX)$ of automorphisms of the affine Dynkin diagram, expressed as permutations of $I$. 
Then $\Daut$ acts on $W_a$ by affine Dynkin automorphisms:
for $\pi\in \Daut$ and $i\in I$, $\pi s_i \pi^{-1}=s_{\pi(i)}$.

To describe the group $\Daut$ explicitly,
say that a node $i\in I$ is \textit{special} if
there is an element of $\Aut$ which sends $0$ to $i$. 
Denote by $I^\spc$ the subset of special nodes.
There is an isomorphism $\Daut \cong P^\vee/Q^\vee$.
There is a bijection $I^\spc\to P^\vee/Q^\vee$ given by $i\mapsto \omega_i^\vee+Q^\vee$
where $\omega_i^\vee$ is the fundamental coweight (and $\omega_0^\vee=0$ by convention). For $i\in I^\spc$,
addition by $\omega_i^\vee+Q^\vee$ on $P^\vee/Q^\vee$ induces a permutation of the set $I^\spc$, which extends uniquely to a permutation of $I$ that defines an element $\pi_i\in\Aut$. Then $\Daut=\{\pi_i\mid i\in I^\spc\}$. We have
\begin{align}\label{E:autotransfinite}
  \pi_i &= t_{\omega_i^\vee} u_{\omega_i^\vee}^{-1}\qquad\text{for $i\in I^\spc$}
\end{align}
where, for $\la\in P^\vee$, $u_\la\in W_0$ is the shortest element such that $u_\la(\la)$ is antidominant.

The dual untwisted case is very similar. We have
$W_a \cong W_0 \ltimes Q$ and $W_e\cong W_0\ltimes P$
where $Q$ and $P$ are the root and weight lattices of the classical subsystem. The rest is entirely similar to the
untwisted case except that coroots and coweights are replaced by roots and weights. The general version of \eqref{E:tmux} is given by \eqref{E:tmuxx} below.

\subsection{Affine Weyl groups for double affine data}
Let $(X,Y)$ be a double affine datum and $(\tX,\tY)$ the
associated pair of affine root data.
In light of \eqref{E:Xbetween} and \eqref{E:Ybetween}
and the identification of the coroot/coweight
lattice with the root/weight lattice of dual finite type,
we have (by definition for $W(\tilde{X})$)
\begin{align}
\label{E:WatX}
W_a(\tX) &= W_0 \ltimes Q^Y \\
\label{E:WtX}
W(\tX) &= W_0 \ltimes Y  \\
\label{E:AuttX}
&= \Daut^X \ltimes W_a(\tX)  \\
\label{E:WetX}
W_e(\tX) &= W_0 \ltimes P^Y  \\
\label{E:AutetX}
&= \Daut^X_e \ltimes W_a(\tX) 
\end{align}
where $\Daut^X \cong Y/Q^Y$ is the restriction of the isomorphism $\Daut^X_e \cong P^Y/Q^Y$. For $i\in I^{X\spc}$ let $\pi_i^X\in \Daut^X_e$ denote the corresponding special automorphism.

The action of $Q^Y\subset W_a(\tX)$ on $R(\tX)$ is given explicitly by
\begin{align}\label{E:tmuxx}
t_\mu(x) = x - (\mu,x)\delta^X \qquad \text{for $\mu\in Q^Y$,
$x\in Q^X$}
\end{align}
and $t_\mu(\de^X)=\de^X$.
Recall that $m\in \Z_{>0}$ was chosen to satisfy $(Y,X)\subset (1/m)\Z$.
We extend \eqref{E:tmuxx} to get an action of $t_\mu$ on $X\oplus\Z(1/m)\de^X$ for $\mu\in Y$, where $t_\mu$ is the corresponding element of $W(\tX)$; this action preserves $R(\tX)$.
From this point on $\tX$ will denote the lattice $X\oplus\Z(1/m)\de^X$. All of the constructions above remain the same for this larger lattice.

We make similar definitions in which $X$ and $Y$ are exchanged. In particular, $s_0^Y \in W(\tY)=X\rtimes W_0$ is given by
\begin{align}\label{E:s0Y}
  s_0^Y = t_\vartheta s_\vartheta
\end{align}
where $\vartheta=\vartheta^X\in X$ is the dominant short root.
We also extend the affine weight lattice to $\tY=Y\oplus\Z(1/m)\de^Y$, as above.

\begin{rem}\label{R:labeling}
The notation $\Daut^X$ serves as a reminder that $\Daut^X$ is a subgroup of $W(\tX)$ and acts on the nodes $I^X$ by automorphisms of the affine Dynkin diagram of $\tX$.
In general, we use a superscript of $X$ or $Y$ to show which (affine) root system a Weyl
group element naturally acts upon. Thus we denote the simple reflections in $W_a(\tX)$ by
$s_i^X$ for $i\in I^X$. Below we will use similar notation to label certain elements in
the DAHA.
\end{rem}

\begin{rem}\label{R:oppositelattice} Note that $X$ is the translation lattice for $W(\tilde{Y})$ and $Y$ is the translation lattice for $W(\tilde{X})$.
\end{rem}

\begin{ex}\label{X:D2trans}
For $(X,Y)=(P(B_2),P(B_2))$ we have $W(\tY)=P(B_2) \rtimes W_0 \cong \Pi^Y_e \ltimes W_a(\tY)$.
$\Pi^Y_e$ is generated by the element $\pi^Y_2$ that acts on $I^Y$ by exchanging $0$ and $2$. $W_a(\tX)$ is as in Example \ref{X:D2Cartan}. 
We have $s_0^Y = t_\vartheta s_\vartheta$ where $\vartheta = e_1$ in the notation of Example \ref{X:D2root}. Since $\vartheta = s_1(\alpha_2)$ we have $s_\vartheta=s_1s_2s_1$ and $t_{-\vartheta} = s_\vartheta s_0^Y =s_1s_2s_1s_0^Y$.
\end{ex}

\begin{lem}\label{L:reducedorbitafex} Let $(X,Y)$ be a
double affine datum. Then
$W_a(\tX)$ and $W_e(\tX)$ have the same orbits on $R_\rd(\tX)$ (consisting of long roots and of short roots)
except for $\tX=A_1^{(1)},C_n^{(1)}, D_{n+1}^{(2)}$. For each such $\tX$, $\Pi_e^X$ is generated by the involution $\pi_n^X$ that reverses the affine Dynkin diagram (and in particular exchanges nodes $0$ and $n$), and the orbit $W_e(\tX)(\alpha_n) = \Z\delta+ W_0 \alpha_n$ is the disjoint union of the orbits
$W_a(\tX)(\alpha_n) =2\Z\delta+W_0\alpha_n$ and
$W_a(\tX)(\alpha_0) =(2\Z+1)\delta+W_0\alpha_n$.
\end{lem}
\begin{proof} One may check that for every irreducible reduced root system of finite type, the subgraph of the Dynkin diagram obtained by restriction to nodes $i$ for $\alpha_i$ short (resp. long) is connected by simple bonds.
Whenever $i,j\in I$ are joined by a simple bond then $s_i s_j \alpha_i=\alpha_j$ and $W_a(\tX)(\alpha_i)=W_a(\tX)(\alpha_j)$.
We deduce that the orbit structure of $W_e(\tX)$ and $W_a(\tX)$ are the same if $0$ is connected to another node by a simple bond. The alternative is that $0$ is connected by a double bond. Since $\tX$ is untwisted or dual untwisted, the only choices for $\tX$ are $A_1^{(1)}$, $C_n^{(1)}$, and $D_{n+1}^{(2)}$. In each case one checks the facts recorded in the Lemma.
\end{proof}

\subsection{Nonreduced affine root systems} Given a double affine datum $(X,Y)$, define the set of \textit{doubled} nodes
\begin{align}
\label{E:Xdoubleset}
  S^X &= \{i\in I^X\mid \alpha_i^{\vee X} \in 2 \tX^\ds \} \\
\label{E:doublesetalt}
  &= \{ i\in I^X \mid \pair{\alpha_i^{\vee X}}{X}\subset 2\Z\}.
\end{align}
The sets \eqref{E:Xdoubleset} and \eqref{E:doublesetalt} are equal because $\pair{\alpha_i^\vee}{\delta}=0$ for all $i\in I$.

\begin{ex} \label{X:red} In the definition of $S^X$ 
the choice of the lattice $X$ between $Q^X$ and $P^X$, is crucial. Suppose $X$ has type $B_n$ and $Y$ has type $C_n$. Then $\tX$ has type $B_n^{(1)}$.
If $X=Q(B_n)$ then $S^X=\{n\}$ but if $X=P(B_n)$ then $S^X=\emptyset$.
\end{ex}

Say that a node $i\in I^X$ has an even Cartan row if $\pair{\alpha_i^\vee}{\alpha_j}\in 2\Z$ for all $j\in I^X$. 

\begin{lem} \label{L:nonred} 
\begin{enumerate}
\item Among affine root systems of untwisted or dual untwisted type, the list of nodes with even Cartan rows is:
$A_1^{(1)}$, nodes $0,1$; $B_n^{(1)}$ for $n\ge 2$, node $n$; $C_2^{(1)}$, node $1$; $D_{n+1}^{(2)}$, nodes $0,n$.
\item $S^X$ is the set of all nodes in $I^X$ with even Cartan rows if $X=Q^X$ and is empty otherwise.
\end{enumerate}
\end{lem}
\begin{proof} For (2), suppose $X=Q^X$. It is immediate from the definitions that if $i\in S^X$ then $i$ has even Cartan row. Conversely, if $i$ has even Cartan row then 
since the simple roots $\{\alpha_j\mid j\in I_0\}$ are a basis of $Q^X=X$, it follows that $i\in S^X$. Now suppose $X\ne Q^X$. We must show that $S^X=\emptyset$. Suppose not, say $i\in S^X$. We have $i\notin I_0$ for otherwise $\omega_i\in P^X=X$ and $\pair{\alpha_i^\vee}{\omega_i}=1\not\in 2\Z$. So $i=0$ is the only possibility. Using (1), $\tX$ either has type $A_1^{(1)}$ or $D_{n+1}^{(2)}$, and in either case $\pair{\alpha_0^\vee}{\omega_1}=-1\not\in 2\Z$ so that $0\notin S^X$, as required.
\end{proof}

\begin{rem}\label{R:nonreddegen} 
$B_2^{(1)}$ is the relabeling of $C_2^{(1)}$ in which the central node is $2$. 
\end{rem}

The not-necessarily-reduced set $R(\tX)$ of roots
is defined by
\begin{align}
\label{R:nonredroots}
R(\tX) = R_\rd(\tX) \cup \bigcup_{i\in S^X} 
W_a(\tX) (2\alpha_i).
\end{align}
$R(\tX)$ is said to be reduced if $R(\tX)=R_\rd(\tX)$ or equivalently if $S^X=\emptyset$.

Similar definitions are made for $R(\tY)$.
Suppose that $(X,Y)$ and $(X',Y')$ are double affine data with $X$ and $X'$ of the same type and $Y$ and $Y'$ of the same type. Say that $(X,Y)$ and $(X',Y')$ are orbit-equivalent if $R(\tX)=R(\tX')$, $R(\tY)=R(\tY')$, the $W(\tX)$-orbits on $R(\tX)$ are the same as the $W(\tX')$-orbits on $R(\tX')$, and similarly for the $Y$s.

\begin{lem}\label{L:rootorbits} 
\begin{enumerate}
\item 
Suppose $(X,Y)$ is a double affine datum such that $\tX$ has no node with even Cartan row. Then $(X,Y)$ and $(P^X,Y)$ are orbit equivalent.
\item Suppose $(X,Y)$ is a double affine datum such that
$\tX$ has a node with even Cartan row. Then $Q^X$ has index $2$ in $P^X$; $(Q^X,Y)$ and $(P^X,Y)$  are not orbit equivalent; and if $X'$ has the same type as $X$ and $X'\ne Q^X$ and $X\ne Q^X$ then $(X,Y)$ and $(X',Y)$ are orbit-equivalent.
\end{enumerate}
Moreover similar statements hold with the roles of $X$ and $Y$ interchanged.
\end{lem}
\begin{proof} This follows from the definitions and Lemmas \ref{L:reducedorbitafex} and \ref{L:nonred}.
\end{proof}

\begin{rem}\label{R:PorQ} In light of Lemma \ref{L:rootorbits}, up to orbit equivalence we may always take $X=P^X$ or $X=Q^X$ (resp. $Y=P^Y$ or $Y=Q^Y$) with the latter only a possibility if $\tX$ (resp. $\tY$) has a node with even Cartan row. 
\end{rem}

\subsection{Hecke parameters}
Let $(X,Y)$ be a double affine datum. Let $\K$ be a field that contains invertible elements $v_\alpha$ for every $\alpha\in R(\tX)$, which depend only on the $W(\tX)$-orbit: $v_\alpha=v_\beta$ if $W(\tX)(\alpha)=W(\tX)(\beta)$. In particular we have elements $v_{\alpha_i}$ for $i\in I^X$ and elements $v_{2\alpha_i}$ for $i\in S^X$. For convenience for $i\in I^X\setminus S^X$ we write $v_{2\alpha_i}=v_{\alpha_i}$. In the literature the symbol $t_i$ is often used for  $v_{\alpha_i}^2$.

\begin{rem}\label{R:mixedtype}
If $2\alpha_i\in R(\tX)$ for some $i\in I$, then specializing $v_{2\alpha_i}=v_{\alpha_i}$
has the effect of deleting the affine Weyl orbit of the root $2\alpha_i$ from the set of
real roots. Specializing $v_{2\alpha_i}=1$ removes the orbit of the root $\alpha_i$.
For instance, consider the Koornwinder double affine datum $(X,Y)=(Q(B_n),Q(B_n))$,
which has $S^X=\{0,n\}=S^Y$.
\begin{enumerate}
\item
To get nonreduced type $A_{2n}^{(2)}$ (with added root $2\alpha_0$), set $v_{2\alpha_n}=1$.
To get reduced $A_{2n}^{(2)}$, further set $v_{2\alpha_0}=v_{\alpha_0}$.
\item To get nonreduced type $A_{2n}^{(2)\dagger}$ (with added root $2\alpha_n$), set $v_{2\alpha_0}=1$. To get reduced $A_{2n}^{(2)\dagger}$, further set $v_{2\alpha_n}=v_{\alpha_n}$.
\item To get reduced type $D_{n+1}^{(2)}$, set 
$v_{2\alpha_0}=v_{\alpha_0}$ and $v_{2\alpha_n}=v_{\alpha_n}$.
\item To get type $C_n^{(1)}$, set 
$v_{2\alpha_0}=v_{2\alpha_n}=1$.
\end{enumerate}
\end{rem}

\subsection{Presentation of DAHA}
\label{SS:dahax}
Let $(X,Y)$ be a double affine datum. The DAHA $\HH(X,Y;\{v_{\alpha_i},v_{2\alpha_i}\})$ 
is the $\K$-algebra with generators $X^\la$ for $\la\in X$, $\pi\in\Daut^X$, $T_i=T_i^X$
for $i\in I^X$, and $q^{\pm 1/m}$ and the following relations. One has
$\pi T_i \pi^{-1} = T_{\pi(i)}$ for $i\in I^X$ and $\pi\in\Daut^X$. Moreover we regard
$\K X[q^{\pm(1/m)}]$ as the group algebra over $\K$ for the lattice
$\tX= X \oplus \Z (1/m)\delta^X$ where $q=X^{\delta^X}$. We have $\pi \mu \pi^{-1} = \pi(\mu)$,
for $\pi\in\Daut^X$ and $\mu\in \tX$, and correspondingly we have the relations
$\pi X^\mu \pi^{-1}=X^{\pi(\mu)}$ in $\HH(X,Y)$. The element $q^{\pm 1/m}$ is central.
The $T_i=T_i^X$ are subject to the same braid relations as the generators $s_i=s_i^X \in W_a(\tX)$
(see Remark~\ref{R:labeling} above for an explanation of the superscript notation). For $w\in W(\tilde{X})$, let $w=\pi^X s_{i_1}\dotsm s_{i_\ell}$ be a reduced expression (one in which $\ell$ is minimal). Define
\begin{align}\label{E:Tw}
  T_w^X = \pi^X T_{i_1}^X \dotsm T_{i_\ell}^X.
\end{align}
Since the $T_i$ satisfy the braid relations this is independent of the reduced expression.

Finally, there are the quadratic relations
\begin{align}\label{E:quadred}
  (T_i - v_{\alpha_i})(T_i + v_{\alpha_i}^{-1}) &= 0
  \qquad\text{for all $i\in I^X$} \\
\label{E:quadnonred}
  (T_i^{-1}X^{-\alpha_i} - v_{2\alpha_i})
  (T_i^{-1}X^{-\alpha_i} + v_{2\alpha_i}^{-1}) &= 0
  \qquad\text{for all $i\in S^X$} 
\end{align}
and the commutation relation
\begin{align}\label{E:commnonred}
  T_i X^\la - X^{s_i(\la)} T_i &=
  \dfrac{(v_{\alpha_i}-v_{\alpha_i}^{-1}) + (v_{2\alpha_i}-v_{2\alpha_i}^{-1})X^{\alpha_i}}{1-X^{2\alpha_i}}(X^\la - X^{s_i(\la)})
\end{align}
for all $\la\in X$ and $i\in I^X$. When $i\in I^X\setminus S^X$, \eqref{E:commnonred} reduces to
\begin{align}\label{E:commred}
  T_i X^\la - X^{s_i(\la)} T_i &=
  \dfrac{(v_{\alpha_i}-v_{\alpha_i}^{-1})  }{1-X^{\alpha_i}}(X^\la - X^{s_i(\la)})
\end{align}
We also note that, when $i\in S^X$, \eqref{E:quadnonred} is equivalent to the special case of
\eqref{E:commnonred} when $\la=\alpha_i$.

\subsection{Dual Hecke parameters}
\label{SS:dualparameters}
The DAHA duality theorem, Theorem \ref{T:dualitydifferent} below, requires another indexing
of the Hecke parameters. Let $j\in I_0$ be the index of any short simple root $\alpha_j\in X$.
Let
\begin{align}\label{E:dualparameters}
  v_{\alpha_i^Y} &= v_{\alpha_i^X} &&\text{if $i\in I_0$} \\
  v_{\alpha_0^Y} &= v_{2\alpha_j^X} \\
  v_{2\alpha_i^Y} &= v_{\alpha_0^X} &&\text{if $i\in S^Y\setminus\{0^Y\}$} \\
\label{E:enddualparameters}
  v_{2\alpha_0^Y} &= v_{2\alpha_0^X} &&\text{if $0^Y\in S^Y$.}
\end{align}

\begin{rem} \cite{St}
The number of $W(\tY)$-orbits on $R(\tY)$
equals the number of $W(\tX)$-orbits on $R(\tX)$.
\end{rem}

\subsection{DAHA duality}
\label{SS:DAHAdual}

Define elements $Y^\la$ in $\HH(X,Y)$ for $\la\in Y$ as follows.
If $\la$ is dominant, then $Y^\la=T_{t_\la}^X$, 
where $t_\la \in W_0 \ltimes Y \cong W(\tilde{X})$; see \eqref{E:Tw}. 

For arbitrary $\la\in Y$, first write $\la = \mu - \nu$, where $\mu$ and $\nu$ are
dominant, and then let $Y^\la = Y^{\mu}(Y^{\nu})^{-1}$. This definition is independent of
the decomposition of $\la$, and the $Y^\la$ generate a subalgebra of
$\HH(X,Y)$ isomorphic to the group algebra $\K Y$.

Let $T_i^Y\in \HH(X,Y)$ be defined by $T_i^Y=T_i^X$ for $i\in I_0$ and 
$$ T_0^Y=(X^\vartheta T_{s_\vartheta})^{-1} $$
where $\vartheta = \vartheta^X$ is the short dominant root in $X$.
Let 
$$ \pi_i^Y = X^{\omega_i^X}T_{u_i^{-1}}$$
for $i\in I^{Y\spc}$ where $u_i=u_{\omega_i^X}$; note that $\pi_0^Y=1$ due to our convention $\omega_0^X=0$. The duality theorem below implies that the elements $\pi_i^Y\in \HH(X,Y)$ for $i\in I^{Y\spc}$ form a group naturally isomorphic to $\Daut^Y=X/Q^X$. Finally, set $Y^{-\delta^Y} = q$.

\begin{thm}\label{T:dualitydifferent} \cite[Cor. 5.12]{Hai}
Let $\epsilon$ be an automorphism of $\K$
such that $\epsilon(v_{\alpha_i^X})=v_{\alpha_i^X}^{-1}$
and $\epsilon(v_{2\alpha_i^X})=v_{2\alpha_i^X}^{-1}$
for all $i\in I^X$. Then there is an $\epsilon$-linear
isomorphism 
$$\delta : \HH(X,Y;\{v_{\alpha_i^X},v_{2\alpha_i^X}\})\to 
\HH(Y,X;\{v_{\alpha_i^Y},v_{2\alpha_i^Y}\})$$
that is the identity on 
$X$, $Y$, $\Daut^X$, and $\Daut^Y$, maps $q$ to $q^{-1}$,
$T_i$ to $T_i^{-1}$ for $i\in I_0$, $T_0^X\mapsto (T_0^X)^{-1}$, $T_0^Y\mapsto (T_0^Y)^{-1}$,
with parameters $v_{\alpha_i^Y}$ and $v_{2\alpha_i^Y}$
defined as in \S \ref{SS:dualparameters}.
\end{thm}

\begin{rem}
In the DAHA $\HH(Y,X)$, the elements $Y^\la$, $\pi_i^Y$, and $T_i^Y$ are generators as in \S~\ref{SS:dahax}, while the elements $X^\la$, $\pi_i^X$, and $T_i^X$ are derived from the generators as in this subsection.
\end{rem}

\begin{rem}
The duality theorem was discovered by Cherednik in the reduced dual untwisted setting \cite[Thm. 2.2]{Che:1992}. We refer to \cite[\S 4.13]{Hai} for a discussion of the history of Theorem~\ref{T:dualitydifferent} and a proof in the not-necessarily-reduced setting.
\end{rem}

\subsection{Intertwiners}
Equation \eqref{E:commnonred} can be rewritten as
\begin{align}\label{E:intertwinernonred}
\phi_i^X X^\la &= X^{s^X_{i,q}(\la)} \phi_i^X
\qquad \text{for all $i\in I^X$ and $\la\in X$}
\end{align}
where
\begin{align}
\label{E:phinonred}
  \phi_i^X &= T_i^X - \dfrac{(v_{\alpha_i}-v_{\alpha_i}^{-1})+(v_{2\alpha_i}-v_{2\alpha_i}^{-1})X^{\alpha_i}}{1-X^{2\alpha_i}} \\
\label{E:phinonredneg}
  &= (T_i^X)^{-1} - \dfrac{(v_{\alpha_i}-v_{\alpha_i}^{-1})X^{2\alpha_i}+(v_{2\alpha_i}-v_{2\alpha_i}^{-1})X^{\alpha_i}}{1-X^{2\alpha_i}}
\end{align}
The notation $s^X_{i,q}$ refers to the action of $s_i\in W_a(\tX)$ on $\tX$. 

The equality of formulas \eqref{E:phinonred} and \eqref{E:phinonredneg} follows from \eqref{E:quadred}.

When $i\in I^X\setminus S^X$, \eqref{E:phinonred} reduces to
\begin{align}
\label{E:phired}
  \phi_i^X &= T_i^X - \dfrac{v_{\alpha_i}-v_{\alpha_i}^{-1}}{1-X^{\alpha_i}} \\
\label{E:phiredneg}
  &= (T_i^X)^{-1} - \dfrac{(v_{\alpha_i}-v_{\alpha_i}^{-1})X^{\alpha_i}}{1-X^{\alpha_i}}.
\end{align}

The $\phi_i^X$ satisfy the same braid relations as the simple reflections $s_i\in W(\tX)$.

\subsection{Dual intertwiners}
Define $\phi_i^Y$ as above, but with $Y$ in place of $X$, i.e., for $\HH(Y,X)$.
Let $\psi_i=\delta^{-1}(\phi_i^Y)$. By Theorem \ref{T:dualitydifferent} we have
\begin{align}\label{E:intertwinerY}
  \psi_i Y^\mu &= Y^{s^Y_{i,q}(\mu)} \psi_i
\end{align}
for $i\in I^Y$ and $\mu\in Y$. Explicitly,
\begin{align}
\label{E:psi}
  \psi_i &= T_i^Y - \dfrac{(v_{\alpha_i^Y}-v_{\alpha_i^Y}^{-1})+(v_{2\alpha_i^Y}-v_{2\alpha_i^Y}^{-1})Y^{-\alpha_i^Y}}{1-Y^{-2\alpha_i^Y}} \\
\label{E:psineg}
  &= (T_i^Y)^{-1} - \dfrac{(v_{\alpha_i^Y}-v_{\alpha_i^Y}^{-1})Y^{-2\alpha_i^Y}+(v_{2\alpha_i^Y}-v_{2\alpha_i^Y}^{-1})Y^{-\alpha_i^Y}}{1-Y^{-2\alpha_i^Y}}.
\end{align}
The notation $s^Y_{i,q}$ refers to the action of $s_i\in W_a(\tY)$ on $\tY$.

The $\psi_i$ satisfy the same braid relations as $s_i\in W_a(\tY)$. For any reduced
expression $w = \pi^Y s_{i_1}\dotsc s_{i_\ell} \in W(\tY)$ where $\pi^Y\in\Daut^Y$ and 
$i_1,\ldots,i_l\in I^Y$, define
\begin{align}\label{E:psiw}
\psi_w = \pi^Y\psi_{i_1}\dotsm \psi_{i_\ell}.
\end{align}
This definition is independent of the reduced expression for $w$.

\subsection{Trivial module for the affine Hecke algebra}
Let $\H(\tX)$ be the $\K$-subalgebra of $\HH(X,Y)$ generated by $T_i^X$ for $i\in I^X$ and
$\pi^X\in \Pi^X$. $\H(\tX)$ is the Hecke algebra for the affine root datum $\tX$. Let $\K \bo$ be the trivial
one-dimensional $\H(\tX)$-module, defined by
\begin{align}
\label{E:Tigen} 
  T_i^X(\bo) &= v_{\alpha_i} \bo \\
\label{E:piXgen}
  \pi^X(\bo) &= \bo.
\end{align}
For $w\in W(\tilde{X})$ let $\Inv(w) = R_+(\tX)\cap - w^{-1} R_+(\tX)$ be the set of right inversions of $w$ (the notation $R_+(\tX)$ disallows doubled roots, that is, 
$R_+(\tX)\subset R_\rd(X)$ by definition). We have
\begin{align}\label{E:Inv}
\Inv(w)=\{s_{i_\ell}\dotsm s_{i_{j+1}} \alpha_{i_j}\mid 1\le j\le \ell\}\qquad\text{for $w=\pi^X s_{i_1}\dotsm s_{i_\ell}$ reduced.}
\end{align}
We have
\begin{align}\label{E:TXeigen}
  T_w^X (\bo) = v_w \bo
\end{align}
where $v_w\in \K$ is defined by
\begin{align}\label{E:vw}
  v_w = \prod_{\beta\in \Inv(w)} v_\beta.
\end{align}

\subsection{Polynomial module}
The {\em polynomial module} of $\HH(X,Y)$ is defined by
\begin{align}\label{E:polymod}
\Pol=\Ind_{\H(\tX)}^{\HH(X,Y)} \K \bo.
\end{align}
The PBW Theorem for the DAHA (see \cite[Thm. 3.2.1]{Che:2005} or \cite[Cor. 5.8]{Hai}) implies that $\Pol=\K X[q^{\pm 1/m}]\bo=\K \tX \bo$.

\subsection{Nonsymmetric Macdonald-Koornwinder polynomials}
For $\la\in X$, let $m_\la \in X \rtimes W_0=W(\tY)$ be the element of minimum length in the coset $t_\la W_0$.
Explicitly 
\begin{align}\label{E:mla}
  m_\la = t_\la u_\la^{-1}
\end{align}
where $u_\la\in W_0$ is the shortest element such that $u_\la(\la)\in -X_+$. The nonsymmetric Macdonald-Koornwinder polynomial $E_\la\in \Pol$ is defined by:
\begin{align}
\label{E:tEdef}
  \tE_\la &= \psi_{m_\la} \bo \\
\label{E:Edef}
  E_\la &= v^{-1}_{u_\la^{-1}} \tE_\la.
\end{align}
We normalize $E_\la$ so that the monomial $X^\la$ has coefficient $1$. This definition takes place in the polynomial module with scalars extended to $\K(q^{1/m})$. However, it is easy to see that the coefficients of $E_\la$ are rational functions of $q$ and the $v_\al^2$.

The use of intertwiners to construct the $E_\la$ originated in \cite{Kn} for type $A$ and \cite{Che:1997} for the general dual untwisted case.

\subsection{Symmetric Macdonald-Koornwinder polynomials}
For $\la\in X_+$, let $(W_0)_\la$ be the stabilizer of $\la$ in $W_0$ and $W_0^\la$ the set of minimum length coset representatives for $W_0/(W_0)_\la$.

The symmetric Macdonald-Koornwinder polynomial 
$P_\la\in \Pol$ is defined by
\begin{align}\label{E:Pdef}
  P_\la = \sum_{u\in W_0^\la} v_u\, T_u(E_\la).
\end{align}
The element $P_\la$ is $W_0$-invariant and the coefficient of the monomial symmetric function $\sum_{u\in W_0^\la} X^{u(\la)}$ in $P_\la$ is equal to $1$.

\begin{rem} \label{R:oversymmetrize} 
\cite[(5.7.5)]{Mac:2003} For $\la\in X_+$ one has
\begin{align}\label{E:Poversym}
  P_\la = 
  \frac{\sum_{u\in W_0} v_u \,T_u(E_\la)}{\sum_{u\in(W_0)_\la} v_u^2}.
\end{align}
\end{rem}

\subsection{Eigenvalues} The nonsymmetric Macdonald-Koornwinder polynomials $E_\la$ are simultaneous eigenvectors for the family of operators $Y^\mu$. In this section we complete the formula in \cite[Prop. 6.9]{Hai} for these eigenvalues.

By Lemma \ref{L:reducedorbitafex}
there are at most three $W(\tX)$-orbits on $R_\rd(\tX)$. The nonempty sets among the following form the $W(\tX)$-orbit decomposition of $R_\rd(\tX)$.
\begin{align}
  R^s(\tX) &= W(\tX)(\alpha_i) \qquad\text{for $i\in I_0$ such that $\alpha_i$ is short} \\
  R^l(\tX) &= W(\tX)(\alpha_j) \setminus R^s(\tX) \qquad\text{for $j\in I_0$ with $\alpha_j$ long} \\
  R^0(\tX) &= W(\tX)(\alpha_0) \setminus (R^s(\tX)\cup R^l(\tX))
\end{align}
These orbits have corresponding Hecke parameters
$v_s = v_\alpha$ for $\alpha$ in the set of short roots $R^s(X)$, $v_l = v_\beta$ for $\beta$ in the set of long roots $R^l(X)$, and $v_0 = v_{\alpha_0}$. 

For $\al\in R(\tX)$ define the symbols $k_\al$ by the logarithmic notation $v_{\alpha} = q^{k_\al}$. Let
\begin{align}
\rho^{\vee Y}_k = (1/2)\sum_{\alpha\in R_+(X)} k_\alpha\, \alpha^{Y\vee}
\end{align}
where, for $\alpha \in R_+(X)$, $\alpha^{Y}\in R_+(Y)$ is defined by $s_{\alpha^Y} = s_\alpha$ in $W_0$.
In the untwisted case, if $\alpha\in R_+(X)$ is short
then $\alpha^\vee$ is long; therefore $\alpha^Y=\eta^{-1}(\alpha^\vee)$ is long
and $\alpha^{Y\vee}$ is short. Similarly, in the untwisted case, if $\alpha$ is long
then $\alpha^{Y\vee}$ is long.
In the dual untwisted case, $Q^Y=Q^X$ and $\alpha^{Y\vee}=\alpha^\vee$.
Let $k_s=k_\alpha$ for short $\alpha$ and $k_l=k_\alpha$ for long $\alpha$. Then we have
\begin{align}
  \rho^{\vee Y}_k &= 
  \begin{cases} k_s \rho^{\vee Y}_s +  k_l \rho^{\vee Y}_l 
  &\text{if untwisted} \\
  k_s \rho^{\vee Y}_l + k_l \rho^{\vee Y}_s & \text{if dual untwisted}
  \end{cases}
\end{align}
where $2\rho^{\vee Y}_s$ (resp. $2\rho^{\vee Y}_l$) is the sum of
the short (resp. long) positive {\em coroots} of $Y$. Note that $v_s=q^{k_s}$ and $v_l=q^{k_l}$.

\begin{prop}\label{P:eigen}
Let $\la\in X$ and $\mu = a\,\delta^Y + \beta\in \tY$ with $\beta\in Y$. Then $E_\la$ is an eigenvector for $Y^\mu$
with eigenvalue $\chi_{\mu,\la} \in \K[q^{\pm1}]$ given by
\begin{align}
  \chi_{\mu,\la} &= \begin{cases}
  \chi_{\mu,\la}^e & \text{if $R^0_\rd(\tX)=\emptyset$} \\
  \chi_{\mu,\la}^s & \text{if $R^0_\rd(\tX)\ne\emptyset$ and $\alpha_n \in R^s(X)$} \\
  \chi_{\mu,\la}^l & \text{if $R^0_\rd(\tX)\ne\emptyset$ and $\alpha_n \in R^l(X)$} \\
  \end{cases}
\end{align}
where
\begin{align}\label{E:chiextended}
  \chi_{\mu,\la}^e = 
  q^{-a-(\mu,\la)+\pair{u_\la^{-1}(2\rho^{\vee Y}_k)}{\beta}}
\end{align}
and $\chi_{\mu,\la}^s$ (resp. $\chi_{\mu,\la}^l$) is obtained from $\chi_{\mu,\la}^e$ by sending $v_s\mapsto (v_sv_0)^{1/2}$ (resp. $v_l\mapsto (v_lv_0)^{1/2}$ and preserving the other parameters.

In particular if all Hecke parameters $v_{\alpha_i}$ and $v_{2\alpha_i}$ are set to a single invertible parameter $v\in \K$ then 
\begin{align}\label{E:eigenequalparam}
  \chi_{\mu,\la} = q^{-a-(\mu,\la)} v^{\pair{u_\la^{-1}(2\rho^{\vee Y})}{\beta}}
\end{align}
where $2\rho^{\vee Y} = \sum_{\alpha^\vee\in R_+^\vee(Y)} \alpha^\vee$.
\end{prop}

\begin{rem}\label{R:subs} If $\chi_{\mu,\la}=\chi_{\mu,\la}^s$ 
(resp. $\chi_{\mu,\la}=\chi_{\mu,\la}^l$)
then the power of $v_s$ (resp. $v_l$) in $\chi_{\mu,\la}^e$ is even so that $\chi_{\mu,\la}$ is a monomial in $q, v_s, v_l, v_0$.
\end{rem}

\begin{rem}\label{R:Haimaneigen} In \cite[Prop. 6.9]{Hai} the formula is stated and proved when $R^0_\rd(\tX)=\emptyset$. 
\end{rem}

\begin{ex} \label{X:orbitdecomp}
Let $(X,Y)=(Q(B_2),Q(B_2))$, $\mu=\beta=\omega_1\in Y$, and $\la=0\in X$. Then $\tX$ is of type $D_{2+1}^{(2)}$ and
$t_\mu \in W(\tX)=Y \rtimes W_0$ has reduced decomposition $t_\mu = s_0 s_1 s_2 s_1$. The inversions of $t_\mu$ are
$\alpha_1$, $s_1\alpha_2=\alpha_1+\alpha_2$, $s_1s_2\alpha_1=\alpha_1+2\alpha_2$, and $s_1s_2s_1\alpha_0=\delta+\alpha_1+\alpha_2$. Since $R_+^s(X)=\{\alpha_2,\alpha_1+\alpha_2\}$ and
$R_+^l(X)=\{\alpha_1,\alpha_1+2\alpha_2\}$, looking directly at the $W(\tX)$-orbits of the inversions, we have $\xi_{\mu,\la} = v_0 v_s v_l^2$. Using the formulas,
$R_+^{\vee s}(Y) = \{\alpha_1^\vee,\alpha_1^\vee+\alpha_2^\vee \}$, $R_+^{\vee l}(X)=
\{ \alpha_2^\vee, 2\alpha_1^\vee+\alpha_2^\vee\}$,
$2\rho_s^{\vee X}=2\alpha_1^\vee+\alpha_2^\vee$, $2\rho_l^{\vee X} = 2\alpha_1^\vee+2\alpha_2^\vee$,
$\pair{2\rho_s^{\vee Y}}{\beta} = 2$, and $\pair{2\rho_l^{\vee Y}}{\beta} = 2$, so that  $\xi_{\mu,\la}^e = v_s^2 v_l^2$
and $\xi_{\mu,\la} = v_0 v_s v_l^2$.
\end{ex}

For $w\in W(\tX)$ let
$\Inv_s(w)=\Inv(w)\cap R^s(\tX)$,
$\Inv_l(w)=\Inv(w)\cap R^l(\tX)$, and
$\Inv_0(w)=\Inv(w)\cap R^0(\tX)$. Then for $w\in W(\tX)$,
\begin{align}\label{E:vwx}
  v_w &= v_0^{|\Inv_0(w)|} v_s^{|\Inv_s(w)|} v_l^{|\Inv_l(w)|}.
\end{align}

We recall some facts about real roots in affine root systems.

\begin{lem}\label{L:affineroots}
\cite{Kac} Let $\tX$ be an irreducible reduced
affine root system of untwisted or dual untwisted type.
Suppose $\beta\in R(X)$ and $a\in\Z$.
\begin{enumerate}
\item If $\tX$ is untwisted then $a\delta+\beta \in R(\tX)$ always, and $a\delta+\beta\in R_+(\tX)$ if and only if $a\in \Z_{>0}$ or if $a=0$ and $\beta\in R_+(X)$.
\item If $\tX$ is dual untwisted then $a\delta+\beta\in R(\tX)$ if and only if $a\in ((\beta,\beta)/2)\Z$ where $(,)$ is the $W_0$-invariant pairing defined above.
Moreover $a\delta+\beta\in R_+(\tX)$ if and only if
$a \in ((\beta,\beta)/2) \Z_{>0}$ or $a=0$ and $\beta\in R_+(X)$.
\end{enumerate}
\end{lem}

\begin{lem} \label{L:tinv}
Suppose $\mu\in Y_+$
and $a\delta+\beta\in R(\tX)$ for $a\in\Z$ and $\beta\in R(X)$. Then 
$a\delta+\beta\in \Inv(t_\mu)$ if and only if $0\le a < (\mu,\beta)$ and $\beta\in R_+(X)$.
\end{lem}
\begin{proof} This follows easily from the formula
\begin{align}
\label{E:transroot}
t_\mu (a\delta+\beta)= \beta + (a - (\mu,\beta))\delta,
\end{align}
which holds for any $\mu\in Y$, $\beta\in X$, and $a\in (1/m)\Z$;
see \eqref{E:tmuxx}.
\end{proof}

Let $R_+^s(X)$ (resp. $R_+^l(X)$) denote the set of short (resp. long) roots in $R_+(X)$ and let $R^{\vee s}_+(X)$
(resp. $R^{\vee l}_+(X)$
be the set of short (resp. long) positive coroots.

Let $2\rho_s^X = \sum_{\alpha\in R_+^s(X)} \alpha$,
$2\rho_l^X = \sum_{\alpha\in R_+^l(X)} \alpha$,
$2\rho_s^{\vee X} = \sum_{\alpha\in R_+^{\vee s}(X)} \alpha$, and
$2\rho_l^{\vee X} = \sum_{\alpha\in R_+^{\vee l}(X)} \alpha$.

\begin{lem} \label{L:tinv0} Suppose $\mu\in Y_+$.
\begin{enumerate}
\item If $R^0(\tX)=\emptyset$, then
$|\Inv_s(t_\mu)|=\pair{2\rho_s^{\vee Y}}{\mu}$ and
$|\Inv_l(t_\mu)|=\pair{2\rho_l^{\vee Y}}{\mu}$ for $\tX$ untwisted and $|\Inv_s(t_\mu)|=\pair{2\rho_l^{\vee Y}}{\mu}$ and
$|\Inv_l(t_\mu)|=\pair{2\rho_s^{\vee Y}}{\mu}$ for $\tX$ dual untwisted.
\item Suppose $R^0(\tX)\ne\emptyset$.
If $\alpha_n$ is short (resp. long) then $|\Inv_0(t_\mu)|=
|\Inv_s(t_\mu)|$ (resp. $|\Inv_0(t_\mu)|=
|\Inv_l(t_\mu)|$), the formula for $|\Inv_l(t_\mu)|$ (resp. $|\Inv_s(t_\mu)|$) is as before, and the formula for
$|\Inv_s(t_\mu)|$ (resp. $|\Inv_l(t_\mu)|$) is $1/2$ the value in the previous case.
\end{enumerate}
\end{lem}
\begin{proof}
Suppose first that $\tX$ is untwisted and $R^0(\tX)=\emptyset$. By Lemmas \ref{L:affineroots} and \ref{L:tinv} the elements of $\Inv_s(t_\mu)$
consist of $a\delta+\beta$ where $\beta \in R_+^s(X)$ and $0\le a < (\mu,\beta)$. It follows that
$|\Inv_s(t_\mu)| = (\mu,2\rho_s^X)=\pair{2\rho_s^{\vee Y}}{\mu}$. The second equality holds by the definition of $(,)$ and the observation that short roots of type $X$ map to short coroots of type $Y$.
Similarly we have $|\Inv_l(t_\mu)| = \pair{2\rho_l^{\vee Y}}{\mu}$.

Suppose next that $\tX$ is dual untwisted and
$R^0(\tX)=\emptyset$. 
By Lemma \ref{L:tinv} the elements of $\Inv_l(t_\mu)$ are precisely given by $a\delta + \beta$ where $\beta\in R_+^l(X)$, $0 \le a < (\mu,\beta)$,
and $a\in \Z$ is a multiple of $r = (1/2)(\beta,\beta)$.
By \eqref{E:pair}
we have $(\mu,\beta) = r \pair{\beta^\vee}{\mu}\in r \Z$.
We deduce that for a fixed long root $\beta\in R_+(X)$, the number of $a$ such that $a\delta+\beta\in\Inv_l(t_\mu)$, is $\pair{\beta^\vee}{\mu}$. In dual untwisted type, $X$ and $Y$ have the same Cartan type. Therefore
$|\Inv_l(t_\mu)| = \pair{2\rho_s^{\vee Y}}{\mu}$ since long roots have short associated coroots.

A similar argument for short roots shows that
$|\Inv_s(t_\mu)|$ is as stated.

Finally, suppose $R^0(\tX)\ne\emptyset$. This happens exactly when $W_a(\tX)$ and $W_e(\tX)$ have different orbit structure on $R(\tX)$. Lemma \ref{L:reducedorbitafex} lists the cases when this occurs. One may check that in all cases, $(P^X,\beta)\in 2\Z$ for $\beta \in W_0(\alpha_n)$. It follows that $|\Inv_0(t_\mu)|
= |\Inv(t_\mu) \cap W(\tX)(\alpha_n)|$, which equals
$|\Inv_s(t_\mu)|$ (resp. $|\Inv_l(t_\mu)|$) if $\alpha_n$ is
short (resp. long). The rest may be deduced using similar arguments.
\end{proof}

\begin{proof}[Proof of Prop. \ref{P:eigen}]
By the definition of $Y^\mu$ in \S \ref{SS:DAHAdual},
the fact that the map $\mu\mapsto \chi_{\mu,\la}$ is a group homomorphism, and that the expressions for $\chi_{\mu,\la}$ also define a group homomorphism, we may assume that $\mu\in Y_+$.

We first consider the case $\la=0$.
Since $E_0=\bo\in\Pol$, by the definition of $Y_\mu$
and \eqref{E:TXeigen} the result reduces to counting the inversions of $t_\mu$ according to $W(\tY)$-orbit, which is accomplished by Lemma \ref{L:tinv0}.

The proof of \cite[Prop. 6.9]{Hai} can be adapted to deduce the result for all $\la\in X$.
\end{proof}

\section{Ram-Yip formula}
We record the details of the Ram-Yip formula for the nonsymmetric Macdonald-Koornwinder polynomials, for not-necessarily-reduced DAHAs. The formula in \cite{RY} applies to the reduced equal-parameter case.

\subsection{The elements $X^w$}
Let $(X,Y)$ be a double affine datum and $\H=\H(X,Y)$.

Let $w\in W(\tY)$. Consider a not-necessarily-reduced factorization
\begin{align}\label{E:arbfact}
w = \pi^Y s_{i_1}s_{i_2}\dotsm s_{i_\ell}
\end{align}
where $\pi^Y\in \Pi^Y$ and $\wa=(i_1,i_2,\dotsc,i_\ell)$ is a sequence of elements in $I^Y$. Define
\begin{align}
\label{E:Xdef}
  X^{\pi^Y;\wa} = \pi^Y \,(T^Y_{i_1})^{\epsilon_1}  (T^Y_{i_2})^{\epsilon_2} \dotsm  (T^Y_{i_\ell})^{\epsilon_\ell} 
\end{align}
where
\begin{align}
\label{E:sign}
\epsilon_k = \begin{cases}
1 &\text{if $\pi^Y s_{i_1}\dotsm s_{i_{k-1}}( \alpha^Y_{i_k}) \in\Z\delta^Y+R_+(Y)$} \\
-1 & \text{if $\pi^Y s_{i_1}\dotsm s_{i_{k-1}}( \alpha^Y_{i_k}) \in\Z\delta^Y-R_+(Y)$.}
\end{cases}
\end{align}

\begin{lem} \cite{Go} \cite{RY}
\label{L:Xv}
For every $w\in W(\tY)$, the element $X^{\pi^Y;\wa}$ is independent of the factorization
$w=\pi^Y s_{i_1}\dotsm s_{i_\ell}$ and hence well-defines an element $X^w\in\HH$. Moreover,
\begin{align}
&\label{E:Xonpoly}
  X^w \bo = v_{\dir{w}} \, X^{\wt(w)}
\end{align}
in $\Pol$, where the elements $\wt(w)\in X$ and $\dir{w}\in W_0$ are defined by the decomposition $W(\tY) \cong X \rtimes W_0$:
\begin{align}
\label{E:enddir}
w = t_{\wt(w)} \dir{w}.
\end{align}
\end{lem}

\begin{rem} 
When $w=t_{\la}$ for $\la\in X$, the element $X^w$ is equal to the generator $X^\la$ of $\HH$.
This is a consequence of Theorem \ref{T:dualitydifferent}.
\end{rem}

\begin{rem}\label{R:W0}
If $w\in W_0$ then $X^w = T_w$. This follows from \eqref{E:Inv}.
\end{rem}

\begin{ex} For any $(X,Y)$ let $i\in I$.
Let $w=\id = s_i s_i$ so $\pi^Y=\id$ and $\wa=(i,i)$.
We have $\pi^Y \alpha_i^Y = \alpha_i^Y$ 
and $\pi^Y s_i \alpha_i^Y = -\alpha_i^Y$ so $\epsilon_1=1$ and $\epsilon_2=-1$.
We have $X^{\id;(1,1)} = (T^Y_i) (T^Y_i)^{-1} = \id$.
\end{ex}

\begin{ex} Let $(X,Y)=(P(A_2),P(A_2))$ so that
$\tX$ and $\tY$ are both of reduced type $A_2^{(1)}$.
There is a single Hecke parameter, which we denote by $v$.
Let $w = t_{\omega_1} = \pi^Y_1 s_2 s_1$, which is reduced.
We have $\wt(w)=\omega_1$, $\dir{w}=\id$,
$\pi^Y_1 \alpha^Y_2 = \alpha^Y_0 = \delta^Y-\alpha_1^Y-\alpha_2^Y$,
$\pi^Y_1 s_2 \alpha^Y_1 = \pi^Y_1 (\alpha^Y_1+\alpha^Y_2)=\alpha^Y_2+\alpha^Y_0 = \delta^Y-\alpha^Y_1$, and
\begin{align*}
  X^{t_{\omega_1}} \bo &= \pi^Y_1 (T^Y_2)^{-1} (T^Y_1)^{-1} \bo = X^{\omega_1}.
\end{align*}
\end{ex}

\begin{ex} Let $(X,Y)$ be as in the previous example.
Let $\la = 2\omega_1 - \omega_2 = \al_1$. Then $u_\la = s_2 s_1$.
Let $w = t_\la s_1 s_2 = s_2 s_0$; this expression is reduced.
We have $s_2(\al_0^Y)=\al_0^Y+\al_2^Y=-\al_1^Y+\de^Y$ and hence
\begin{align*}
X^w \bo &= T_2^Y (T_0^Y)^{-1} \bo = T_2 X^\vartheta T_{s_\vartheta} \bo
= X^{s_2(\vartheta)} T_2^{-1} T_{s_\vartheta} \bo = v^2 X^{\al_1}
\end{align*}
where $\vartheta=\vartheta^X=\theta^X=\al_1+\al_2$.
This agrees with $\wt(w)=\la=\al_1$ and $\dir{w}=s_1 s_2$.
\end{ex}

\comment{
\begin{ex} Let $(X,Y)$ be as in the previous example.
Let $\la = 3\omega_1-\omega_2$. Then $u_\la = s_2s_1$.
Let $w = t_\la s_1 s_2 = \pi^Y_1 s_1 s_2 s_1 s_0$;
this expression is reduced. We have
$\pi^Y_1 \alpha^Y_1 = \alpha^Y_2$, 
$\pi^Y_1 s_1 \alpha^Y_2 = \delta^Y-\alpha^Y_1$,
$\pi^Y_1 s_1 s_2 \alpha^Y_1 = \delta^Y-\alpha^Y_1-\alpha^Y_2$,
$\pi^Y_1 s_1 s_2 s_1 \alpha^Y_0 = 2\delta^Y-\alpha^Y_1$.
\begin{align*}
  X^w \bo &= \pi^Y_1 (T_1^Y) (T^Y_2)^{-1} (T^Y_1)^{-1} (T^Y_0)^{-1} \bo \\
  &= v^2 X^\la,
\end{align*}
agreeing with $\wt(w)=\la$ and $\dir{w}=s_1s_2$.
\end{ex}
}

\subsection{Ram-Yip DAHA formula, preliminary version}
\label{SS:RYDAHA}
The goal of this subsection is to state a form of \cite[Thm. 2.2]{RY} for arbitrary DAHAs.

Let $(X,Y)$ be a double affine datum and $w\in W(\tY)$.
Let $w=\pi^Y s_{i_1}\dotsm s_{i_\ell}$ be a reduced factorization and write $\wa=(i_1,\dotsc,i_\ell)$.
Define the sequence of affine real roots $\beta_k\in R(\tY)$ by
\begin{align}
\label{E:betadef}
\beta_k=\beta_k(\wa) = s_{i_\ell} \dotsm s_{i_{k+1}} \alpha_{i_k}^Y\qquad\text{for $1\le k\le \ell$.}
\end{align}
These comprise the set $\Inv(w)=\Inv((\pi^Y)^{-1}w)$
by \eqref{E:Inv}.

\begin{ex}\label{X:setup} With the double affine datum of Example \ref{X:D2trans}, let $w=t_{-\vartheta}$. We have $\pi^Y=\id$ and reduced word $\wa=(1,2,1,0)$, and inversions
$\beta_1=2\delta-(e_1-e_2)$, $\beta_2=2\delta-e_1$,  $\beta_3=2\delta-(e_1+e_2)$, $\beta_4=\delta-e_1$.
\end{ex}

Fix an element $u\in W(\tY)$.

Let $b=(b_1,b_2,\dotsc,b_\ell) \in \{0,1\}^\ell$ be a binary word of length $\ell$. It is used to select an arbitrary subsequence of reflections from $\wa$.
Define the elements $u_k\in W(\tY)$ by
\begin{align}
\label{E:u0}
  u_0 &= u \pi^Y \\
\label{E:udef}
  u_k &= u_{k-1} s_{i_k}^{b_k}\qquad\text{for $1\le k\le \ell$} \\
\label{E:uend}
  \ep(b) &= u_\ell.
\end{align}
Define the signs $\epsilon_k\in \{\pm 1\}$ by
\begin{align}
\label{E:signsforpath}
  \epsilon_k = \begin{cases}
  1 & \text{if $u_{k-1} \alpha_{i_k}^Y \in \Z\delta^Y + R_+(Y)$} \\
  -1 & \text{if $u_{k-1} \alpha_{i_k}^Y \in \Z\delta^Y - R_+(Y)$.} \\
  \end{cases}
\end{align}
If $u=\id$ and $b=(1,\ldots,1)$ this agrees with the previous definition of $\epsilon_k$ from \eqref{E:sign}.

\begin{ex}\label{X:D2path} Take $w$ and $\wa$ from Example \ref{X:setup}
and $u=\id$. Let $b=(0,1,1,0)$. Then $(u_0,\dotsc,u_4)=(\id,\id,s_2, s_2s_1,s_2s_1)$ and $(\epsilon_1,\dotsc,\epsilon_4)=(1,1,1,1)$.
\end{ex}

Define the elements $c_k=c_k(\wa)$ and $d_k=d_k(\wa) \in \K$ for $1\le k\le \ell$ by
\begin{align}
\label{E:firstscalar}
c_k &= v_{\alpha_{i_k}^Y}^{-1} - v_{\alpha_{i_k}^Y} \\
\label{E:secondscalar}
d_k &= v_{2\alpha_{i_k}^Y}^{-1} - v_{2\alpha_{i_k}^Y} 
\end{align}

\begin{thm} \label{T:tauexpDAHA} \cite[Thm. 2.2]{RY}
\begin{align}
  X^u \psi_w &= \sum_{b=(b_1,\dotsc,b_\ell)} X^{\ep(b)} \prod_{b_k=0} \begin{cases} \dfrac{c_k +d_kY^{-\beta_k}}{1-Y^{-2\beta_k}} &\text{if $\epsilon_k=1$} \\
  \dfrac{c_k Y^{-2\beta_k} +d_kY^{-\beta_k}}{1-Y^{-2\beta_k}} & \text{if $\epsilon_k=-1$}
  \end{cases}
\end{align}
\end{thm}
\begin{proof}
The idea is to expand the intertwiners from left to right, using \eqref{E:psi} or \eqref{E:psineg} 
depending on how previous choices were made. We first have
\begin{align}\label{E:Xupi}
X^u \pi^Y = X^{u \pi^Y}.
\end{align}
This relation follows from the definition \eqref{E:Xdef} and the commutation
\begin{align}
\label{E:gTcomm}
\pi^Y T^Y_i = T^Y_{\pi^Y(i)} \pi^Y \qquad\text{for $i\in I^Y$}
\end{align}
which holds by Theorem \ref{T:dualitydifferent} and the corresponding defining relation of the DAHA $\H(Y,X)$. By \eqref{E:Xupi} we have
\begin{align*}
 X^u  \pi^Y \psi_{i_1} \psi_{i_2}\dotsm \psi_{i_\ell} &=\
 X^{u \pi^Y} \psi_{i_1} \psi_{i_2}\dotsm \psi_{i_\ell}.
\end{align*}
We expand $\psi_{i_1}$ depending on $\epsilon_1$ (see \eqref{E:signsforpath}).
\begin{enumerate}
\item If $\epsilon_1=1$, by \eqref{E:psi} and \eqref{E:intertwinerY} we have
\begin{align*}
  X^{u\pi^Y} \psi_{i_1} \psi_{i_2}\dotsm \psi_{i_\ell}  &= 
  X^{u\pi^Y} \left(T_{i_1}^Y + \dfrac{c_1 + d_1 Y^{-\alpha_{i_1}}}{1-Y^{-2\alpha_{i_1}}} \right)\psi_{i_2}\dotsm \psi_{i_\ell} \\
  &= X^{u\pi^Y s_{i_1}} \psi_{i_2}\dotsm \psi_{i_\ell} +
  X^{u\pi} \psi_{i_2}\dotsm \psi_{i_\ell} \dfrac{c_1+d_1Y^{-\beta_1}}{1-Y^{-2\beta_1}} \\
\end{align*}
\item If $\epsilon_1 = -1$, by \eqref{E:psineg} and \eqref{E:intertwinerY} we have
\begin{align*}
  X^{u\pi^Y} \psi_{i_1} \psi_{i_2}\dotsm \psi_{i_\ell}  &= 
  X^{u\pi^Y} \left((T_{i_1}^Y)^{-1} + \dfrac{c_1 Y^{-2\alpha_{i_1}}+d_1 Y^{-\alpha_{i_1}}}{1-Y^{-2\alpha_{i_1}}}\right) \psi_{i_2}\dotsm \psi_{i_\ell} \\
  &= X^{u\pi^Y s_{i_1}} \psi_{i_2}\dotsm \psi_{i_\ell} +
  X^{u\pi^Y} \psi_{i_2}\dotsm \psi_{i_\ell} \dfrac{c_1Y^{-2\beta_1}+d_1 Y^{-\beta_1}}{1-Y^{-2\beta_1}} \\
\end{align*}
\end{enumerate}
The Theorem follows by expanding $X^{u\pi^Y s_{i_1}} \psi_{s_{i_2}\dotsm s_{i_\ell}}$
and $X^{u\pi^Y} \psi_{s_{i_2}\dotsm s_{i_\ell}}$ by induction.
\end{proof}

\begin{ex} Let $(X,Y)=(P(A_2),P(A_2))$ and write $v$ for all parameters. 
Let $u=\id$ and $w=\pi^Y_1 s_1 s_0$.
We have $\beta_1=s_0\alpha_1=\alpha_0+\alpha_1 = \delta-\alpha_2$
and $\beta_2 = \alpha_0 = \delta -\alpha_1-\alpha_2$. We first compute by direct expansion
following the proof of Theorem~\ref{T:tauexpDAHA}. We use that $\pi^Y_1 \alpha_1 = \alpha_2$,
$\pi^Y_1 s_1 \alpha_0 = \pi_1 (\alpha_0+\alpha_1)=\alpha_1+\alpha_2$,
$\pi^Y_1 \alpha_0 = \alpha_1$. Thus we use three cases of \eqref{E:psi}.
\begin{align*}
X^{\id} \psi_{\pi^Y_1  s_1s_0}
        &= \pi^Y_1 \psi_1 \psi_0 \\
        &= \pi^Y_1 \left(T_1 + \dfrac{v^{-1}-v}{1-Y^{-\alpha_1}}\right)\psi_0  \\
        &= \pi^Y_1 T_1 \psi_0  + \pi^Y_1 \psi_0 \dfrac{v^{-1}-v}{1-Y^{-s_0(\alpha_1)}}  \\
        &= \pi^Y_1 T_1 \left(T_0^Y + \dfrac{v^{-1}-v}{1-Y^{-\alpha_0}}\right)+
        \pi^Y_1 \left(T_0^Y + \dfrac{v^{-1}-v}{1-Y^{-\alpha_0}}\right) \dfrac{v^{-1}-v}{1-Y^{-s_0(\alpha_1)}} \\
        &= \pi^Y_1 T_1 T_0^Y  + \pi^Y_1 T_1 \dfrac{v^{-1}-v}{1-Y^{-\alpha_0}} \\
        &+ \pi^Y_1 T_0^Y \dfrac{v^{-1}-v}{1-Y^{-s_0(\alpha_1)}}  +
        \pi^Y_1 \dfrac{v^{-1}-v}{1-Y^{-\alpha_0}}\dfrac{v^{-1}-v}{1-Y^{-s_0(\alpha_1)}}  
\end{align*}
We list $(b_1,b_2)$, $(u_1,u_2)$, and the corresponding terms in $X^\id \psi_{s_1s_0}$. We have $u_0=\pi^Y_1$.
In every case it turns out that $\epsilon_1=\epsilon_2=1$.
\begin{itemize}
\item $(1,1)$: $(u_1,u_2)=(\pi^Y_1 s_1, \pi^Y_1 s_1s_0)$ with term $X^{\pi s_1s_0}$.
\item $(1,0)$: $(u_1,u_2)=(\pi^Y_1 s_1, \pi^Y_1 s_1)$ with term
$X^{\pi^Y_1 s_1} (v^{-1}-v)(1-Y^{-\beta_2})^{-1}$.
\item $(0,1)$: $(u_1,u_2)=(\pi^Y_1,\pi^Y_1 s_0)$ with term $X^{\pi^Y_1 s_0} (v^{-1}-v)(1-Y^{-\beta_1})^{-1}$.
\item $(0,0)$: $(u_1,u_2)=(\pi^Y_1,\pi^Y_1)$ with term $X^{\pi^Y_1} (v^{-1}-v) (1-Y^{-\beta_1})^{-1} (v^{-1}-v) (1-Y^{-\beta_2})^{-1}$.
\end{itemize}
\end{ex}

\subsection{Ram-Yip formula, working version} 
\label{SS:RYwork}
Recall that we started with $w, u\in W(\tY)$, which gave rise to $\pi^Y$ and $\wa$. Given a binary word $b=(b_1,\dotsc,b_\ell)$, let $J=J(b)=\{j\mid b_j=0\}$ be the subset of $\{1,2,\dotsc,\ell\}$ of positions where $b$ has \textit{zeroes}. For any subset $J=\{j_1<\dotsm<j_r\}\subset\{1,2,\dotsc,\ell\}$,
define $z_0,\dotsc,z_r\in W(\tY)$ by
\begin{align}\label{E:zdef}
z_0 &= u w \\
z_m &= z_{m-1} s_{\beta_{j_m}}\qquad\text{for $1\le m\le r$.}
\end{align}
By what will eventually be abuse of language, we call this data the alcove path $p_J$. This data is depicted
\begin{align}\label{E:zpath}
z_0 \xrightarrow{\beta_{j_1}} z_1 \xrightarrow{\beta_{j_2}} \dotsm \xrightarrow{\beta_{j_r}} z_r =: \ep(p_J)
\end{align}
Equivalently, $z_0=uw$, and successive $z$'s are obtained by removing the $j_1$-th reflection, then the $j_2$-th, and so on. We have
\begin{align}\label{E:signz}
  \epsilon_{j_m} = 
  \begin{cases} 1 & \text{if $z_m \beta_{j_m} \in \Z \delta^Y + R_+(Y)$} \\
  -1 & \text{if $z_m\beta_{j_m}\in \Z\delta^Y-R_+(Y)$.}
  \end{cases}
\end{align}
This holds because $z_m \beta_{j_m} = u_{j_m-1} \alpha_{i_{j_m}}$. Define the subsets $J^\pm \subset J$ by
\begin{align}\label{E:Jpm}
  J^\pm = \{ j\in J \mid \epsilon_j = \pm 1 \}.
\end{align}
We denote by $\BP(u;w)=\{p_J\mid J\subset \{1,2,\dotsc,\ell\} \}$ (by abuse of notation, since it depends on the factorization of $w$) the set of such sequences $(z_0,z_1,\dotsc,z_r)$ (together with the knowledge of $J$).

\begin{ex}\label{X:D2pathz} Following Example \ref{X:D2path} and its binary word $b$, we have $J=\{1,4\}$, $(z_0,z_1,z_2)=(s_1s_2s_1s_0,s_2s_1s_0,s_2s_1)$, $J^+=\{1,4\}$, and $J^-=\emptyset$.
\end{ex}

In the new notation we have the equality in $\HH$:
\begin{align}
\label{E:RY}
X^u \psi_w &=
\sum_{p_J\in \BP(u;w)}
X^{\ep(p_J)} \prod_{j\in J^+} \dfrac{c_j + d_j Y^{-\beta_j}}{1-Y^{-2\beta_j}}
\prod_{j\in J^-} \dfrac{c_jY^{-2\beta_j} + d_j Y^{-\beta_j}}{1-Y^{-2\beta_j}}.
\end{align}

\subsection{Alcove paths} 
We explain the name ``alcove path" for $p_J$.
This is called an LS-gallery in \cite{GL}.

For $i\in I_0$ let $H_i$ be the hyperplane in 
$X_\R = X \otimes_{\Z} \R$ through the origin, normal to $\alpha_i^X$. Let $H_0$ be the hyperplane through the point $\vartheta/2$, normal to $\vartheta$ (recall \eqref{E:s0Y}). Then $s_i^Y\in W_a(\tY)$ is the reflection across the hyperplane $H_i$ for $i\in I$.
The complement in $X_\R$ of the union of hyperplanes $wH_i$
for $w\in W_a(\tY)$ and $i\in I$, has components called alcoves. Say that a hyperplane $H'$ is a wall of an alcove $A'$ if the boundary of $A'$ contains a nonempty Euclidean-open subset of $H'$. 
The fundamental alcove $A$ is the unique alcove that has $H_i$ as a wall for all $i\in I$. There is a bijection from $W_a(\tY)$ to the set of alcoves in $X_\R$ given by $w\mapsto wA$.
The alcove $wA$ has walls $w H_i$ for $i\in I$.
Given $i\in I^Y$, the alcoves $wA$ and $ws_i A$ are on opposite sides of their one common wall $wH_i=ws_iH_i$.

This describes the alcoves for $W_a(\tY)$.
The group $W(\tY) \cong \Pi^Y \ltimes W_a(\tY)$ acts on the set $\Pi^Y \times X_\R$ by $(\pi' w) (\pi, x) = 
(\pi'\pi, \pi^{-1}w\pi(x))$. This set has a connected component $\pi \times X_\R$ (the $\pi$ sheet) for every $\pi\in\Pi^Y$. The group $\Pi^Y$ permutes the sheets and $W_a(\tY)$ has a natural action on the identity sheet and a $\pi$-twisted action on the $\pi$ sheet. The alcoves of $\Pi^Y \times X_\R$ are by definition the sets $\pi \times w A$ where $\pi\in \Pi^Y$ and $w\in W_a(\tY)$. There is a bijection from $W(\tY)$ to the set of alcoves in $\Pi^Y \times X_\R$ given by $u\mapsto \pi \times wA$ where $u=\pi w$ with $\pi \in\Pi^Y$ and $w\in W_a(\tY)$. We will write $uA$ instead of $\pi \times wA$. Then for all $u\in W(\tY)$ and
$i\in I$, $uA$ and $us_iA$ are on opposite sides of their unique common wall $uH_i=us_iH_i$.

Given the data $u, w\in W(\tY)$
let $w=\pi^Y s_{i_1}\dotsm s_{i_\ell}$ be reduced
with $\pi^Y\in \Daut^Y$ and $\wa=(i_1,\dotsc,i_\ell)$ a sequence in $I^Y$.
Given a binary word $b=(b_1,\dotsc,b_\ell)$,
let $(u_0,u_1,\dotsc,u_\ell)$ be as in \eqref{E:u0} and \eqref{E:udef}. This gives rise to a sequence of alcoves
$(u_0A,u_1A,\dotsc,u_\ell A)$ in the $\pi^Y$ sheet. There is a piecewise-linear continuous parametrized path in the sheet $\pi^Y \times X_\R$ defined informally as follows. The path starts in the center of alcove $u_0A$. For $1\le k\le \ell$, the $k$-th part of the path starts at the center of the alcove $u_{k-1}A$. It follows a line segment towards the 
the center of the alcove $u_{k-1} s_{i_k}A$ until it touches the wall $u_{k-1} H_{i_k}$.
If $b_k=0$ (so that $u_k =u_{k-1}$)
the path bounces back and traces back to the center of $u_k A=u_{k-1}A$. If $b_k=1$ then the path continues transversely across the wall, to the center of the alcove $u_k A$.

The sign $\epsilon_k$ indicates whether the
$k$-th part of the path starts in the positive or negative
orientation with respect to the wall it approaches.
If $b_k=0$, the $k$-th step is called a {\em fold} and if $b_k=1$ it is called a crossing.

The alcove path $p_J$ can also be constructed by ``folding" as follows. Begin with the path $p_\emptyset$. It consists of directed line segments connecting the centers of the alcoves $(u\pi^YA,u\pi^Y s_{i_1}A,u\pi^Y s_{i_1}s_{i_2}A,\dotsc,uwA)$.

To form $p_{\{j_1\}}$, consider the hyperplane crossed by the $j_1$-th line segment of $p_{\emptyset}$. Reflect the part of $p_{\emptyset}$ after the hyperplane, across the hyperplane. This reflection is $s_{\beta_1}$. To form $p_{\{j_1,j_2\}}$, consider the hyperplane crossed by the $j_2$-th segment of $p_{\{1\}}$. Reflect the part of the path $p_{\{j_1\}}$ after the hyperplane, across it (using $s_{\beta_2}$). Continue until $p_J$ is constructed.

\begin{ex} \label{X:alcove} 
Following Example \ref{X:D2pathz}, let
$X_\R = \R e_1 \oplus \R e_2$, $\alpha_1,\alpha_2$ as in Example \ref{X:D2root}, and $\vartheta=e_1$. $H_1$ is the line $x=y$, $H_2$ is the line $y=0$, and $H_0$ is the line $x=1/2$.
The fundamental alcove $A$ for $W(\tY)$ is the open triangle with vertices $(0,0)$, $(1/2,0)$, and $(1/2,1/2)$. 
If an alcove $xA$ appears in a picture for some $x\in W(\tY)$, we draw $i$ on each wall $xH_i$ of the alcove $xA$.

The alcove walk $p_{\{1,4\}}$ is constructed by folding.
We start with $p_{\emptyset}$; see Figure~1.
\begin{figure}
\begin{pspicture}(-2,1)(2,5)
\psset{unit=2cm}
\rput(0,.8){(0,0)}
\rput(1.4,2){(1/2,1/2)}
\rput(.85,1.2){id}
\rput(.5,1.8){$s_1$}
\rput(-.5,1.8){$s_1s_2$}
\rput(-.65,1.15){$s_1s_2s_1$}
\rput(-1.4,1.15){$s_1s_2s_1s_0$}
\psnode(0,1){o}{$\bullet$}
\psnode(1,2){a}{}
\psnode(1,1){b}{}
\psnode(0,2){c}{}
\psnode(-1,2){d}{}
\psnode(-1,1){e}{}
\ncline{o}{a}
\ncput[npos=.3]{1}
\ncline{o}{b}
\ncput{2}
\ncline{a}{b}
\ncput{0}
\ncline{o}{c}
\ncput{2}
\ncline{o}{d}
\ncput[npos=.3]{1}
\ncline{o}{e}
\ncput{2}
\ncline{a}{c}
\ncput{0}
\ncline{c}{d}
\ncput{0}
\ncline{d}{e}
\ncput{0}
\psnode(-2,1){z}{}
\ncline{e}{z}
\ncput{2}
\ncline{d}{z}
\ncput{1}
\psnode(.7,1.3){id}{}
\psline[linecolor=red]{->}(.75,1.25)(.25,1.75)
\psline[linecolor=red]{->}(.25,1.75)(-.25,1.75)
\psline[linecolor=red]{->}(-.25,1.75)(-.75,1.25)
\psline[linecolor=red]{->}(-.75,1.25)(-1.25,1.25)
\end{pspicture}
\caption{Alcove path $p_\emptyset$}
\label{F:ap}
\end{figure}
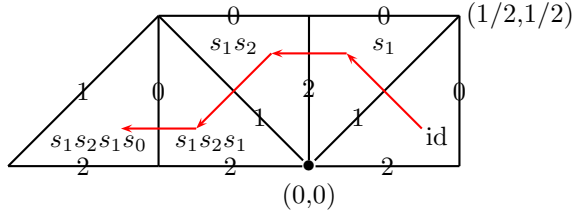
We add a fold in the $j_1=1$-th segment. This reflects
the tail of the path $p_{\emptyset}$ across the line $x=y$, producing the alcove path $p_{\{1\}}$, pictured in Figure \ref{F:ap1}. 
\begin{figure}
\begin{pspicture}(0,-2)(2,4.5)
\psset{unit=2cm}
\rput(-.3,1){(0,0)}
\rput(1.4,2){(1/2,1/2)}
\rput(.85,1.2){id}
\rput(.85,.8){$s_2$}
\rput(.55,.2){$s_2s_1$}
\rput(.3,-.37){$s_2s_1s_0$}
\psnode(0,1){o}{$\bullet$}
\psnode(1,2){a}{}
\psnode(1,1){b}{}
\psnode(1,0){c}{}
\psnode(0,0){d}{}
\ncline{o}{a}
\ncput[npos=.3]{1}
\ncline{o}{b}
\ncput[npos=.3]{2}
\ncline{o}{c}
\ncput[npos=.3]{1}
\ncline{o}{d}
\ncline{d}{c}
\ncput{0}
\ncline{a}{b}
\ncput{0}
\ncline{b}{c}
\ncput{0}
\ncline{d}{o}
\ncput{2}

\psline(0,1)(1,1)
\psnode(.7,1.3){id}{}
\psnode(.7,.7){s2}{}
\psnode(.3,.3){s21}{}
\psnode(.5,1.5){h}{}
\psline[linecolor=red]{->}(.75,1.25)(.5,1.5)
\psline[linecolor=red]{->}(.45,1.45)(.70,1.2)
\psline[linecolor=red]{->}(.70,1.2)(.7,.7)
\psline[linecolor=red]{->}(.7,.7)(.3,.3)
\psline[linecolor=red]{->}(.3,.3)(.3,-.3)
\psnode(0,-1){low}{}
\ncline{d}{low}
\ncput{2}
\ncline{low}{c}
\ncput[npos=.3]{1}
\end{pspicture}
\caption{Alcove path $p_{\{1\}}$}
\label{F:ap1}
\end{figure}
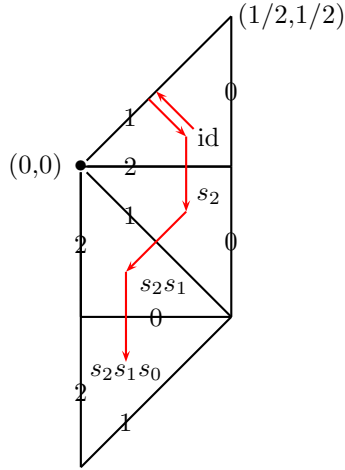
Finally we add a fold to the $j_2=4$-th segment. This reflects the tail of $p_{\{1\}}$ across the line $y=-1/2$, resulting in $p_{\{1,4\}}$, pictured in Figure \ref{F:ap14}. This is the alcove path
of Example \ref{X:D2pathz}. 

\begin{figure}
\begin{pspicture}(0,0)(2,4.5)
\psset{unit=2cm}
\rput(-.3,1){(0,0)}
\rput(1.4,2){(1/2,1/2)}
\rput(.85,1.2){id}
\rput(.85,.8){$s_2$}
\rput(.55,.2){$s_2s_1$}
\psnode(0,1){o}{$\bullet$}
\psnode(1,2){a}{}
\psnode(1,1){b}{}
\psnode(1,0){c}{}
\psnode(0,0){d}{}
\ncline{o}{a}
\ncput[npos=.3]{1}
\ncline{o}{b}
\ncput[npos=.3]{2}
\ncline{o}{c}
\ncput[npos=.3]{1}
\ncline{o}{d}
\ncline{d}{c}
\ncput{0}
\ncline{a}{b}
\ncput{0}
\ncline{b}{c}
\ncput{0}
\ncline{d}{o}
\ncput{2}

\psline(0,1)(1,1)
\psnode(.7,1.3){id}{}
\psnode(.7,.7){s2}{}
\psnode(.3,.3){s21}{}
\psnode(.5,1.5){h}{}
\psline[linecolor=red]{->}(.75,1.25)(.5,1.5)
\psline[linecolor=red]{->}(.45,1.45)(.70,1.2)
\psline[linecolor=red]{->}(.70,1.2)(.7,.7)
\psline[linecolor=red]{->}(.7,.7)(.3,.3)
\psline[linecolor=red]{->}(.3,.3)(.3,0)
\psline[linecolor=red]{->}(.25,0)(.25,.3)
\end{pspicture}
\caption{Alcove path $p_{\{1,4\}}$}
\label{F:ap14}
\end{figure}
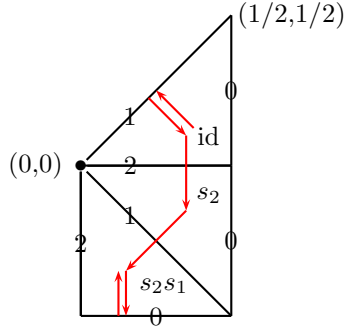
\end{ex}

It is convenient to organize the alcove paths into a tree in which each vertex is an alcove path $p_J$ for a subset $J\subset\{1,2,\dotsc,\ell\}$. The root of the tree is $p_\emptyset$. For $\emptyset\ne J\subset \{1,2,\dotsc,\ell\}$ the parent of $p_J$ is $p_{\hJ}$ where $\hJ$ is obtained from $J$ by removing its largest element $M=\max(J)$ (that is, its last fold) and the directed edge from $\hJ$ to $J$ is labeled $\beta_M$. 
See Example \ref{X:D32red} for an example of a subtree of this tree.

\subsection{Ram-Yip formula for nonsymmetric Macdonald-Koornwinder polynomials}
The following theorem is the not-necessarily-reduced analogue of \cite[Thm. 3.1]{RY}. It is obtained by applying the formula \eqref{E:RY} to the generator of the polynomial module.

\begin{thm} \label{T:TMac} Let $u\in W(\tY)$ and
$\la\in X$. Let $w=m_\la = \pi^Y s_{i_1}\dotsm s_{i_\ell}\in W(\tilde{Y})$ be reduced with $\wa=(i_1,\dotsc,i_\ell)$. Let $\beta_j$ be as in \eqref{E:betadef}, $J^\pm$ as in \eqref{E:Jpm}, and $c_j$ and $d_j$ as in \eqref{E:firstscalar} and \eqref{E:secondscalar}. Finally, with $\xi_j$ defined by $Y^{-\beta_j} \bo = \xi_j \bo$, we have
\begin{align}\label{E:RYE}
  X^u (\tE_\la) &= \sum_{p_J\in \BP(u; m_\la)} X^{\wt(p_J)} v_{\dir{p_J}}
 \prod_{j\in J^+} \dfrac{c_j + d_j \xi_j}{1-\xi_j^2}
 \prod_{j\in J^-} \dfrac{c_j\xi_j^2  + d_j \xi_j}{1-\xi_j^2}
\end{align}
where $\wt(p_J)=\wt(\ep(p_J))$ and $\dir{p_J}=\dir{\ep(p_J)}$. 
\end{thm}

Taking $u=\id$ yields the Ram-Yip formula for $\tE_\la$.

\begin{rem}\label{R:Yeigenvalue} By  Proposition~\ref{P:eigen}: $\xi_j = \chi_{-\beta_j,0}$. In the equal-parameter case, by \eqref{E:eigenequalparam}, $\xi_j = q^{\deg(\beta_j)} v^{\pair{2\rho^{\vee Y}}{-\bbeta_j}}$ where $\deg$ is defined in \eqref{E:deg} and $\bbeta_j$ is the projection of $\beta_j$ to the classical weight lattice.
\end{rem}

\begin{rem}\label{R:RYreduced} If $v_{2\alpha_i}=v_{\alpha_i}$ for all $i\in I$, then $c_j=d_j$ and \eqref{E:RYE} simplifies to
\begin{align}\label{E:RYreduced}
  X^u(\tE_\la) &=
  \sum_{p_J\in \BP(u;m_\la)}
  X^{\wt(p_J)} v_{\dir{p_J}} \prod_{j\in J} \dfrac{v^{-1}_{\alpha_{i_j}}- v_{\alpha_{i_j}}}{1-\xi_j}
  \prod_{j\in J^-} \xi_j
\end{align}
If all $v_{\alpha_i}$ and $v_{2\alpha_i}$ are set to an invertible element $v\in \K$ (the equal-parameters specialization) then \eqref{E:RYE} further simplifies to
\begin{align}\label{E:RYequalparam}
  X^u(\tE_\la) &=
  \sum_{p_J\in \BP(u;m_\la)}
  X^{\wt(p_J)} v^{\ell(\dir{p_J})} (v^{-1}-v)^{|J|} \prod_{j\in J} (1-\xi_j)^{-1} \prod_{j\in J^-} \xi_j
\end{align}
with $\xi_j$ as in Remark \ref{R:Yeigenvalue}.
\end{rem}

\begin{rem}\label{R:RYerratum}
Theorem 3.1 in \cite{RY} is stated only for $\la\in Q^X$, the root lattice of $X$. We also point out an erratum in \cite{RY} related to the formula $(T_0^Y)^{-1}=X^{\vartheta}T_{s_\vartheta}$, where $\vartheta$ is the dominant {\em short} root in $X$. 
$T_0^Y$ is denoted $T_0^\vee$ in \cite[(2.27)]{RY} and the formula given there uses the highest root of $X$, which is equal to the correct root $\vartheta$ if and only if $X$ is simply-laced.
\end{rem}

\begin{rem}\label{R:tlainversions}
 It is well-known that if $w\in W(\tilde{Y})$ is Grassmannian, then $\Inv(w) \subset \Z_{>0}\delta^Y-R_+(Y)$. In particular, the inversions $\beta_j$ of $w=m_\la$ satisfy
\begin{align}
\label{E:GrassInv}
\text{$\be_j\in\Z_{>0}\de^Y-R_+(Y)$ for $j=1,\ldots,\ell$,}
\end{align}
and each $\xi_j$ is a positive power of $q$ times a nonempty product of positive powers of Hecke parameters.
\end{rem}

\subsection{Mixed type}
\label{SS:mixedRY}
We present the Ram-Yip formula for nonsymmetric Macdonald polynomials of type $A_{2n}^{(2)}$ (which has $\alpha_0$ as the extra short root) and its affine dual $A_{2n}^{(2)\dagger}$. Following \cite[\S 6.7]{Hai}, these are obtained from the nonsymmetric Koornwinder polynomials $E_\la^K$ by specialization of Hecke parameters.

Let $\HH_K$ (K is for Koornwinder) be the DAHA with double affine datum $(X,Y)=(Q(B_n),Q(B_n))$. It has $\tX$ and $\tY$ both of type $D_{n+1}^{(2)}$ and $S^X=\{0,n\}$.
There are five independent Hecke parameters: $v_s = v_{\alpha_n^X}$ the short root parameter, $v_l=v_{\alpha_i^X}$ for $1\le i\le n-1$ the long root parameter, $v_0=v_{\alpha_0^X}$ the $\alpha_0$ parameter, $v_2 = v_{2\alpha_n^X}$ the $2\alpha_n$ parameter, and $v_z=v_{2\alpha_0^X}$ the $2\alpha_0$ parameter. 
According to \S \ref{SS:dualparameters} we have
$v_{\alpha_i^Y}=v_l$ for $1\le i\le n-1$,
$v_{\alpha_n^Y}=v_s$, $v_{\alpha_0^Y}=v_2$,
$v_{2\alpha_n^Y} = v_0$, and $v_{2\alpha_0^Y}=v_z$.

For $\la\in X=Q(B_n)$ let $E_\la^K(X;q;v_s,v_l,v_0,v_2,v_z)$ denote the nonsymmetric
Koornwinder polynomials, which are the ``nonsymmetric Macdonald polynomials" for $\HH_K$. The dual intertwiners for $\HH_K$ are given below, where roots are of type $\tY=D_{n+1}^{(2)}$ and $1\le i\le n-1$. Using \eqref{E:psi} and \eqref{E:psineg} for the intertwiners and 
\eqref{E:dualparameters} through \eqref{E:enddualparameters}
for the dual parameters, we obtain
\begin{align}
\label{E:K0+}
\psi^K_0 &= T_0^Y + \dfrac{(v_2^{-1}-v_2)+(v_z^{-1}-v_z)Y^{-\alpha_0^Y}}
{1-Y^{-2\alpha_0^Y}} \\
\label{E:K0-}
&=(T_0^Y)^{-1} + \dfrac{(v_2^{-1}-v_2)Y^{-2\alpha_0^Y}+(v_z^{-1}-v_z)Y^{-\alpha_0^Y}}
{1-Y^{-2\alpha_0^Y}} \\
\label{E:Kl+}
\psi_i^K &= T_i^Y + \dfrac{v_l^{-1}-v_l}{1-Y^{-\alpha_i^Y}} \\
\label{E:Kl-}
&= (T_i^Y)^{-1} + \dfrac{(v_l^{-1}-v_l)Y^{-\alpha_i^Y}}{1-Y^{-\alpha_i^Y}} \\
\label{E:Kn+}
\psi^K_n &= T_n^Y + \dfrac{(v_s^{-1}-v_s)+(v_0^{-1}-v_0)Y^{-\alpha_n^Y}}
{1-Y^{-2\alpha_n^Y}} \\
\label{E:Kn-}
&= (T_n^Y)^{-1} + \dfrac{(v_s^{-1}-v_s)Y^{-2\alpha_n^Y}+(v_0^{-1}-v_0)Y^{-\alpha_n^Y}}
{1-Y^{-2\alpha_n^Y}}.
\end{align}

\subsubsection{$A_{2n}^{(2)}$}
\label{SSS:a2n2}
The DAHA $\HH({A_{2n}^{(2)}})$ of type $A_{2n}^{(2)}$ is 
by definition the image of the Koornwinder DAHA $\HH_K$ (having double affine datum $(Q(B_n),Q(B_n))$) under the specialization $v_{2\alpha_0^X}\mapsto v_{\alpha_0^X}$
and $v_{2\alpha_n^X}\mapsto 1$, that is, $v_z\mapsto v_0$ and $v_2\mapsto 1$. The substitution $v_2=1$ has the effect of getting rid of the $W(\tX)$-orbit of the root $\alpha_n$
and $v_z=v_0$ gets rid of the $W(\tX)$-orbit of the root  $2\alpha_0$, leaving the real roots of type $A_{2n}^{(2)}$.

Let $\la\in P(C_n)$. We regard $\la$ as an element of $Q(B_n)$ via the isomorphism
$P(C_n)\cong Q(B_n)$ given by $\omega_i^{C_n}\mapsto \omega_i^{B_n}$ for $1\le i\le n-1$ and $\omega_n^{C_n}\mapsto 2\omega_n^{B_n}$.
The definition of the nonsymmetric Macdonald polynomial 
of type $A_{2n}^{(2)}$ is 
\begin{align}
  E_\la^{A_{2n}^{(2)}}(X;q;v_s,v_l,v_0) =
  E_\la^K(X;q;v_s,v_l,v_0,v_2=1,v_z=v_0).
\end{align}
Similar definitions are made for $\tE$.

In the equal parameters case we further set $v_s=v$, $v_l=v$, and $v_0=v$.
After these specializations, using \eqref{E:K0+} through \eqref{E:Kn-} we have
\begin{align}
\label{E:psi0a2}
\psi_0^K&\mapsto T_0^Y + \dfrac{(v^{-1}-v) Y^{-\alpha_0^Y}}{1-Y^{-2\alpha_0^Y}} =(T_0^Y)^{-1} + \dfrac{(v^{-1}-v)Y^{-\alpha_0^Y}}{1-Y^{-2\alpha_0^Y}} \\
\label{E:psia2}
\psi_i^K&\mapsto T_i^Y + \dfrac{v^{-1}-v}{1-Y^{-\alpha_i^Y}} =
(T_i^Y)^{-1} + \dfrac{(v^{-1}-v)Y^{-\alpha_i^Y}}{1-Y^{-\alpha_i^Y}} 
\qquad\text{for $i \in I_0$}.
\end{align}
where the roots are of type $\tY=D_{n+1}^{(2)}$.

\begin{prop}
\label{P:RYA2n2}
Let $u\in W(\tY)$ and $\la\in X=Q(B_n)$.
In $\tY=X\rtimes W_0$, let $m_\la = s_{i_1}\dotsm s_{i_\ell}$ be reduced, $\wa=(i_1,\dotsc,i_\ell)$ and $\beta_j$ as before.
For $J\subset \{1,\dotsc,\ell\}$ define
$J^0 = \{j\in J\mid i_j=0\}$ and let
$J^\pm$ be the set of $j\in J\setminus J^0$ where the fold at $j$ has sign $\pm 1$. We have
\begin{align}
\label{E:nsmacA2n2ep}
  X^u E_\la^{A_{2n}^{(2)}}(X;q;v) &=
  \sum_{p_J\in \BP(u;m_\la)}
  X^{\wt(p_J)} v^{-\ell(u_\la^{-1})+\ell(\dir{p_J})-|J|}(1-v^2)^{|J|} \\ \notag
  &\,\,\quad
  \prod_{j\in J^0}\dfrac{\xi_j}{1-\xi_j^2}\prod_{j\in J^+} \dfrac{1}{1-\xi_j} \prod_{j\in J^-} \dfrac{\xi_j}{1-\xi_j} 
\end{align}
where $2\rho^{\vee Y}$ is the sum of positive coroots of type $B_n$ and $\xi_j=q^{\deg(\beta_j)} v^{\pair{2\rho^{\vee Y}}{-\bbeta_j}}$.
\end{prop}

\subsubsection{$A_{2n}^{(2)\dagger}$}
\label{SSS:a2n2dag}
The DAHA $\HH({A_{2n}^{(2)\dagger}})$ of type $A_{2n}^{(2)\dagger}$ is  
by definition the image of $\HH_K$ under the specialization $v_{2\alpha_0^X}\mapsto 1$
and $v_{2\alpha_n^X}\mapsto v_{\alpha_n^X}$, that is, $v_z\mapsto 1$ and $v_2\mapsto v_s$.

The classical subsystem has type $B_n$.
Let $\la\in Q(B_n)$. The definition of the nonsymmetric Macdonald polynomial of type $A_{2n}^{(2)\dagger}$ is 
\begin{align}
  E_\la^{A_{2n}^{(2)\dag}}(X;q;v_s,v_l,v_0) =
  E_\la^K(X;q;v_s,v_l,v_0,v_2=v_s,v_z=1).
\end{align}
Similar definitions are made for $\tE$.
The substitution $v_z=1$ has the effect of getting rid of the $W(\tX)$-orbit of the root $\alpha_0$
and $v_z=v_0$ gets rid of the $W(\tX)$-orbit of $2\alpha_n$, leaving the system of real roots of $A_{2n}^{(2)\dagger}$.

Further setting Hecke parameters all equal to $v$, from \eqref{E:K0+} through \eqref{E:Kn-} we have
\eqref{E:psia2} for $i\in I_0$ and
\begin{align}
\label{E:psi0a2dag}
\psi_0^K&\mapsto T_0^Y + \dfrac{(v^{-1}-v) }{1-Y^{-2\alpha_0^Y}} \\
&= (T_0^Y)^{-1} + \dfrac{(v^{-1}-v)Y^{-2\alpha_0^Y}}{1-Y^{-2\alpha_0^Y}}.
\end{align}

\begin{prop}
\label{P:RYA2n2dag}
Let $u\in W(\tY)$ and $\la\in X=Q(B_n)$. With notation
as in Proposition \ref{P:RYA2n2}, let $J^{0\pm}$ be the subset of $J^0$ consisting of folds with sign $\pm1$. Then
\begin{align}
\label{E:nsmacA2n2dagep}
  X^u E_\la^{A_{2n}^{(2)\dagger}}(X;q;v) &=
  \sum_{p_J\in \BP(u;m_\la)}
  X^{\wt(p_J)} v^{-\ell(u_\la^{-1})+\ell(\dir{p_J})-|J|}(1-v^2)^{|J|} \\ \notag
  &\,\,\quad
  \prod_{j\in J^{0+}} \dfrac{1}{1-\xi_j^2}
  \prod_{j\in J^{0-}}\dfrac{\xi_j^2}{1-\xi_j^2}\prod_{j\in J^+} \dfrac{1}{1-\xi_j} \prod_{j\in J^-} \dfrac{\xi_j}{1-\xi_j}
\end{align}
\end{prop}

\subsection{Ram-Yip for symmetric Macdonald-Koornwinders}
Let $\la\in X_+$, $(W_0)_\la$ be the stabilizer of $\la$ in $W_0$, and $W_0^\la$ the set of minimum length coset representatives in $W_0/(W_0)_\la$. The following is the not-necessarily-reduced analogue of \cite[Thm. 3.4]{RY}, which computes $\tilde{P}_\la$ instead of $P_\la$.

\begin{prop}\label{P:RYP} Let $\la\in X_+$. With notation as in Theorem \ref{T:TMac},
\begin{align}\label{E:RYP}
  P_\la(X;q;v_\bullet)   &= \sum_{u\in W_0^\la} v_{uu_\la^{-1}}^{-1}
  \sum_{p_J(u)\in \BP(u; m_\la)}
   X^{\wt(p_J(u))} v_{\dir{p_J(u)}} f_J(u)
\end{align}
where, for $u\in W_0^\la$, $f_J(u)$ denotes the product over $J^+$ times the product over $J^-$ in \eqref{E:RYE} for the alcove path $p_J(u)$.
\end{prop}
\begin{proof} The Proposition follows by applying Theorem \ref{T:TMac} to the summands of \eqref{E:Pdef} with the help of Remark \ref{R:W0}.
\end{proof}

\section{At $v\to0$}
We study the Ram-Yip formula when $v\to0$, that is, when all the Hecke parameters $v_{\alpha_i},v_{2\alpha_i}$ are set to a single variable $v$ which is then sent to zero.

\subsection{Quantum Bruhat graph for untwisted and dual untwisted affine type} Let $\tX$ be a reduced root datum of untwisted or dual untwisted affine type. Let $(X,Y)$ be the reduced double affine datum that gives rise to the pair $(\tX,\tY)$, where $X=P^X$ is the classical weight lattice of $\tX$ and $Y=P^Y$ is the weight lattice of type dual to (resp. the same as) $X$ if $\tX$ is untwisted (resp. dual untwisted). The quantum Bruhat graph $\QB(\tX)$ of $\tX$ is the directed edge-labeled graph with vertex  set $W_0(Y)\cong W_0(X)$ and directed labeled edges of the form $w\xrightarrow{\alpha} w s_\alpha$ for $w\in W_0(Y)$ and $\alpha\in R_+(Y)$, such that either
\begin{itemize}
\item $\ell(w s_\alpha) = \ell(w) + 1$, giving a covering relation in the usual Bruhat order, or
\item $\ell(w s_\alpha) = \ell(w) - \pair{2\rho^{\vee Y}}{\alpha} + 1$, in which
case the corresponding edge is called a {\em quantum edge}.
\end{itemize}
Here $2\rho^{\vee Y}=\sum_{\alpha^\vee\in R_+^\vee(Y)} \alpha^\vee$.

\begin{rem} \label{R:QB} Let us write $\QB_X$ for the quantum Bruhat graph in the literature \cite{BFP} \cite{LS} \cite{LNSSS}. It is associated with a finite root datum $X$: its vertices have labels $W_0(X)$ and its edges are labeled by elements of $R_+^\vee(X)$.

For $\tX$ untwisted, $X$ and $Y$ are of dual types, $\tX$ is the untwisted affinization of $X$, and $\QB(\tX)\cong \QB_X$ given by $W_0(Y)\cong W_0(X)$ via $s_i^Y\mapsto s_i^X$ for all $i\in I$, and labels correspond via
$\eta:Q^Y \cong Q^{\vee X}$.

For $\tX$ dual untwisted, $X$ and $Y$ have the same type, $X$ is the classical subsystem of $\tX$, and
$\QB(\tX)\cong\QB_{X^\vee}$ where $X^\vee$ is of type dual to that of $X$. The vertices correspond under $W_0(Y)\cong W_0(X^\vee)$ and the labels correspond under the canonical isomorphism $Q^Y \cong Q^{\vee X^\vee}$.

In the dual untwisted case the properties of the
quantum Bruhat graph are being studied in \cite{LNSSS2}.
\end{rem}

\subsection{The order of a term in the Ram-Yip formula}
\label{SS:ord}

With the notation of Theorem \ref{T:TMac}, for $p_J\in \BP(u; m_\la)$, define the affine root lattice element $\qwt(p_J)\in \bigoplus_{i\in I^Y} \Z \alpha_i^Y$ by
\begin{align}
\label{E:qwtdef}
\qwt(p_J) = \sum_{j\in J^-} \beta_j.
\end{align}
Let $\ord(p_J)\in\Z$ denote the order in $v$ of the summand of $p_J$ in the equal-parameter specialization of \eqref{E:RYE} (when all Hecke parameters $v_{\alpha_i}$ and $v_{2\alpha_i}$ have been set to $v$).

\begin{lem}\label{L:order} For all $p_J\in \BP(u; m_\la)$
\begin{align}
\ord(p_J) =  \ell(\dir{p_J}) - |J| - \pair{2\rho^{\vee Y}}{\qwt(p_J)}.
\end{align}
\end{lem}
\begin{proof} This follows directly from Remarks \ref{R:Yeigenvalue}, \ref{R:RYreduced} and \ref{R:tlainversions}.
\end{proof}

Recall the elements $(z_0,\dotsc,z_r)$ in \eqref{E:zdef} associated with $p_J\in\BP(u;m_\la)$
where $J=\{j_1<\dotsm<j_r\}$. Say that $p_J\in \BPQ(u;m_\la)$ (and say $p_J$ is a quantum alcove path) if
\begin{align}
  \dir{z_0} \overset{-\bbeta_{j_1}}{\longrightarrow} \dir{z_1}
  \overset{-\bbeta_{j_1}}{\longrightarrow}\dotsm
  \overset{-\bbeta_{j_r}}{\longrightarrow} \dir{z_r} 
\end{align}
is a path in $\QB(\tX)$.
Here $\bbeta$ is by definition the image of the
affine root $\beta$ in the classical root lattice.
By Remark \ref{R:tlainversions} $-\bbeta_j\in R_+(Y)$ and $\deg(\beta_j)>0$ for all $1\le j\le \ell$.

\begin{prop} \label{P:quantum} Every alcove walk
$p_J\in \BP(u;m_\la)$ has $\ord(p_J) \ge \ord(p_\emptyset)$
with equality if and only if
$p_J\in \BPQ(u;m_\la)$.
\end{prop}
\begin{proof}
The proof proceeds by induction on $|J|$.
For $J=\emptyset$ the induction hypothesis holds trivially.
Otherwise let $J\subset \{1,2,\dotsc,\ell\}$ be nonempty
with maximum element $j_r=M$ and let $\hJ=J\setminus\{M\}$.
By induction $\ord(p_{\hJ})\ge0$
with equality if and only if $p_{\hJ}\in\BPQ(u;m_\la)$.
We have 
\begin{align}
  \ord(p_J) - \ord(p_{\hJ}) = 
  \ell(\dir{(p_J)}) - \ell(\dir{(p_{\hJ})}) - 1 
  - \chi \pair{2\rho^{\vee Y}}{\beta_j}.
\end{align}
Here $\chi=1$ if $M$ is a negative fold and $\chi=0$ otherwise. 

We have $u_{M-1}\alpha_{i_M}=z_r\be_M$ and $\dir{p_J} = \dir{p_{\hJ}}s_{-\bbeta_M}$.
Also, recall that $-\bbeta_M\in R_+(Y)$ due to \eqref{E:GrassInv}.
\begin{itemize}
\item If $M$ is a positive fold of $p_J$ then $\overline{z_r\be_M}\in R_+(Y)$. Equivalently
$$\dir{p_J}(-\overline{\beta}_M)=-\overline{z_r \beta_M}\in -R_+(Y).$$ It follows that
\begin{align}\label{E:posfoldineq}
  \dir{p_J} > \dir{p_{\hJ}}
\end{align}
and therefore $\ord(p_J) \ge  \ord(p_{\hJ})$ with
equality if and only if $\dir{p_J}$
is a Bruhat cover of $\dir{p_{\hJ}}$. 
\item If $M$ is a negative fold of $p_J$ then $\overline{z_r\be_M}\in -R_+(Y)$.
Hence
\begin{align}\label{E:posfoldineq2}
  \dir{p_J} < \dir{p_{\hJ}}
\end{align}
and
\begin{align}\label{E:posfoldlen}
  \ell(\dir{p_J})-\ell(\dir{p_{\hJ}}) \ge 
  -\ell(s_{\bbeta_M}) 
  \ge
  1 - 2\pair{\rho^{\vee Y}}{-\beta_M}.
\end{align}
Therefore, $\ord(p_J)\ge \ord(p_{\hJ})$
with equality if and only if
$\dir{p_{\hJ}} \overset{-\bbeta_M}{\longrightarrow} \dir{p_J}$ is a quantum edge.
\end{itemize}
\end{proof}

\subsection{Specialization at $v=0$}

\begin{cor}\label{C:quantum}
For any $u\in W(\tY)$ and any $\la\in X$,
let $m_\la = \pi^Y s_{i_1}\cdots s_{i_\ell}$ be a reduced expression in $W(\tY)$, and $\wa = (i_1,\ldots,i_\ell)$. Then
\begin{align}\label{E:t0limit}
  \lim_{v\to 0} v^{-1}_{\dir{u} u_\la^{-1}} X^u \tE_\la = 
  \sum_{p\in \BPQ(u; m_\la)} X^{\wt(p)} q^{\deg(\qwt(p))}
\end{align}
Here $v\to 0$ means that all Hecke parameters are set equal to $v$ and then $v$ is sent to $0$.
\end{cor}
\begin{proof}
Since $\ep(p_\emptyset)=um_\la=u t_\la u_\la^{-1} = t_{u(\la)} u u_\la^{-1}$, we have $\ord(p_\emptyset) = \ell(\dir{p_\emptyset}) = \ell(u u_\la^{-1})$. Since every alcove walk $p_J$ satisfies $\ord(p_J)\ge \ord(p_\emptyset)$, the limit in \eqref{E:t0limit} is well-defined. Terms corresponding to $p_J$ not in $\BPQ(u;m_\la)$ have $\ord(p_J)>\ord(p_\emptyset)$ and therefore vanish in the limit. 
\end{proof}

\comment{
\begin{rem}
Formula \eqref{E:t0limit} remains valid when the Hecke parameters tend to $0$ independently,
as long as $v_{\al_i}/v_{2\al_i}\to 1$.
In general, one may define \fixit{Dan, the entirely general case might have up to 5 parameters.}
\begin{align}
\ord_k(p_J) = \ell_k(\dir{p_J})
- \sum_{j\in J} k_{\bbeta_j^X} - \pair{\qwt(p_J)}{\rho_k^{\vee Y}}
\end{align}
as an element of $\Z k_s \oplus \Z k_l$ where $\ell_k(w)=\sum_{\al\in R_+(X)\cap -w^{-1}R_+(X)}k_\al$
for $w \in W_0$ and $\bbeta_j^X\in R_+(X)$ satisfies $s_{\bbeta_j^X}=s_{\bbeta_j}$.
A straightforward refinement of the proof of Proposition~\ref{P:quantum}
shows that the coefficients of $k_s$ and $k_l$ in $\ord_k(p_J)$ are greater than or equal to those
in $\ord_k(p_\emptyset)$, with both being equal 
if and only if $p_J\in \BPQ(u;m_\la)$.
\end{rem}
}

\begin{ex} \label{X:D32red} Let $\wa$ and $\beta_j$ be as in Example \ref{X:setup} and let $u=\id$ and $\la=-\vartheta=-e_1$.
Writing $e_1=(1,0)$ and $e_2=(0,1)$ we have
\begin{align*}
  E_{(-1,0)}^{D_{2+1}^{(2)}}(X;q;0) = X^{(-1,0)} + X^{(0,-1)} + X^{(0,1)}  + X^{(1,0)} + (q+1)X^{(0,0)}.
\end{align*}
For the Ram-Yip formula we have the reduced word $\wa=(1,2,1,0)$ for 
$t_\la =  m_\la = s_1s_2s_1s_0 \in W(\tY)=W_a(D_3^{(2)})$. We have $\beta_1=2\delta-(1,-1)$, $\beta_2=2\delta-(1,0)$, $\beta_3=2\delta-(1,1)$, $\beta_4=\delta-(1,0)$.
The quantum alcove paths are as follows.
\begin{equation*}
\begin{array}{|c||c|c|c|c|} \hline
J & \text{signs} & \mathrm{dir} & \mathrm{wt} & q^{?} \\ \hline \hline
\emptyset & ++++ & 1 & (-1,0) & q^0 \\ \hline
1 & ++++ & s_1 & (0,-1) & q^0 \\ \hline
12  & +++- & s_2s_1 & (0,1) & q^0 \\ \hline
14 & ++++ & s_2s_1 & (0,0) & q^0 \\ \hline
123 & +++- & s_1s_2s_1 & (1,0)&q^0 \\ \hline
1234 & +++- & 1 & (0,0) & q^1 \\ \hline
\end{array}
\end{equation*}
Here is the tree of alcove paths, but restricted to the
elements of $\BPQ(u;m_\la)$. The vertex $p_J$ has been labeled by $\dir{p_J}$ on the top and $J$ on the bottom. The arrows are labeled by $\beta_j$ rather than $-\bbeta_j$.
\begin{align*}
\begin{diagram}
\node{\begin{matrix}
1 \\ \emptyset
\end{matrix}
} \arrow{e,t}{\beta_1}
\node{\begin{matrix}
s_1 \\ 1
\end{matrix}
}\arrow{e,t}{\beta_2}
\arrow{s,t}{\beta_4}
\node{\begin{matrix}
s_2s_1 \\ 12
\end{matrix}
}
\arrow{e,t}{\beta_3}
\node{\begin{matrix}
s_1s_2s_1 \\ 123
\end{matrix}
}
\arrow{e,t}{\beta_4}
\node{\begin{matrix}
1 \\ 1234
\end{matrix}
} \\
\node[2]{
\begin{matrix}
s_2s_1 \\ 14
\end{matrix}
}
\end{diagram}
\end{align*}
There is one quantum edge going from $p_{123}\to p_{1234}$ and the other edges are Bruhat.
\end{ex}

\begin{cor}\label{C:symmact0} For $\la\in X_+$ we have
\begin{align}
  P_\la(X;q;0) 
  &= \sum_{u\in W_0^\la} \sum_{p_J\in \BPQ(u;m_\la)} X^{\wt(p_J)} q^{\deg(\qwt(p_J))}.
\end{align}
\end{cor}
\begin{proof} Follows from \eqref{E:Pdef}
and Corollary \ref{C:quantum}.
\end{proof}

\subsection{Reduced $A_{2n}^{(2)}$ at $v=0$}
\label{SS:A2}

Consider Proposition \ref{P:RYA2n2}.
After sending $v\to0$, all the denominators tend to $1$ and the only different behavior
that enters the Ram-Yip formula comes from \eqref{E:psi0a2}:
a positive fold $j\in J^0$ contributes $\xi_j$ instead of the usual $1$.

With the notation of Proposition \ref{P:RYA2n2},
for $p_J\in \BP(u; m_\la)$ (of type $D_{n+1}^{(2)}$) define
\begin{align}
\label{E:nwta2}
\qwt_{A_{2n}^{(2)}}(p_J) = \sum_{j\in J^0\cup J^-} \beta_j.
\end{align}

\begin{lem}\label{L:ordera2} 
The order of $v$ in the summand of $p_J$ in the formula \eqref{E:nsmacA2n2ep} is
\begin{align}
\ord(p_J) =  -\ell(u_\la^{-1}) + \ell(\dir{p_J}) - |J| - \pair{2\rho^{\vee Y}}{\qwt_{A_{2n}^{(2)}}(p_J)}.
\end{align}
\end{lem}

Let $\BPQ_{A_{2n}^{(2)}}(u;m_\la)$ be the set of quantum alcove paths $p_J\in \BPQ_{D_{n+1}^{(2)}}(u;m_\la)$
whose Bruhat edges never come from the simple reflection $\alpha_0$, or equivalently, the corresponding root $\beta_j$ is not in $(2\Z+1)\delta^Y+R^s(Y)$ where $R^s(Y)$ are the short roots.

\begin{prop}\label{P:A2} We have
\begin{align}\label{E:A2E}
 \lim_{v\to0} v^{\ell(u_\la^{-1})-\ell(\dir{u} u_\la^{-1})}  X^u
  E_\la^{A_{2n}^{(2)}}(X;q,v) = \sum_{p\in \BPQ_{A_{2n}^{(2)}}(u; m_\la)}
  X^{\wt(p)} q^{\deg(\qwt_{A_{2n}^{(2)}}(p))}.
\end{align}
\end{prop}
\begin{proof} The proof is similar to that for Proposition \ref{P:quantum} except paths $p_J$ with a step $j\in J^0$. We have $\beta_j \in (2\Z+1)\delta-R_+(Y)$ by  Lemma \ref{L:reducedorbitafex} and Remark \ref{R:tlainversions}.
If the sign of the fold at $j$ is negative then the proof goes as before for the case of a quantum edge in the quantum Bruhat graph. If the sign of the fold at $j$ is positive then when this fold gets added to the path, the change in $\qwt_{A_{2n}^{(2)}}$ is $-\pair{2\rho^{\vee Y}}{\beta_j}>0$ so that $\ord(p_J)$ cannot be minimum. Because positive folds in paths of type $D_{n+1}^{(2)}$ of minimum order, correspond to Bruhat steps in $\QB(D_{n+1}^{(2)})$, the Proposition follows.
\end{proof}

\begin{ex}\label{X:a2} Let $\la=-\omega_1\in P(C_2)$, the classical weight lattice of $A_{2\cdot2}^{(2)}$. 
It can be viewed as an element in $X$ where $(X,Y)=(Q(B_2),Q(B_2))$ is the Koornwinder double affine datum. 
The computation of $E_\la(X;q;0)$ in type $A_{2n}^{(2)}$ consists of the computation of $E_\la(X;q;0)$ in type $D_{2+1}^{(2)}$ as in Example \ref{X:D32red}, but  with some branches of the tree truncated: the Bruhat edges for roots $\beta_j$ of the form $(2\Z+1)\delta+\beta$ for $\beta$ short, or equivalently, with $i_j=0$. In this example the only such root is $\beta_4=\delta-(1,0)$ and the only Bruhat edge labeled by this root is the edge $p_{1}\to p_{14}$. The answer is
\begin{align*}
  E_{(-1,0)}^{A_{2\cdot2}^{(2)}}(X;q;0) = X^{(-1,0)} + X^{(0,-1)} + X^{(0,1)}  + X^{(1,0)} + q X^{(0,0)}.
\end{align*}
\end{ex}

\subsection{Reduced $A_{2n}^{(2)\dagger}$}
\label{SS:A2dag}

Consider Proposition \ref{P:RYA2n2dag}.
The corresponding change in the Ram-Yip formula from $D_{n+1}^{(2)}$ is that if $j\in J^{0-}$ then the contribution to $\qwt$ will be $2\beta_j$ instead of $\beta_j$.

\begin{align}
\label{E:nwta2dag}
\qwt_{A_{2n}^{(2)\dag}}(p_J) = \sum_{j\in J^-} \beta_j
+ \sum_{\substack{j\in J^0 \\ \text{$j$ is a negative fold}}} 2\beta_j.
\end{align}

\begin{lem}\label{L:ordera2dag} 
The order of $v$ in the summand of $p_J$ in \eqref{E:nsmacA2n2dagep} is given by
\begin{align}
\ord(p_J) = -\ell(u_\la^{-1}) + \ell(\dir{p_J}) - |J| - \pair{2\rho^{\vee Y}}{\qwt_{A_{2n}^{(2)\dagger}}(p_J)}.
\end{align}
\end{lem}

Let $\BPQ_{A_{2n}^{(2)\dag}}(u;m_\la)$ be the set of quantum alcove paths $p_J\in \BPQ_{D_{n+1}^{(2)}}(u;m_\la)$
whose quantum edges never come from the simple reflection $\alpha_0$, or equivalently, the corresponding root $\beta_j$ is not in $(2\Z+1)\delta^Y+R^s(Y)$ where $R^s(Y)$ are the short roots.

\begin{prop}\label{P:A2dag} 
For reduced type $A_{2n}^{(2)\dag}$ with all Hecke parameters sent to $0$,
\begin{align}\label{E:A2dagE}
 \lim_{v\to0} v^{\ell(u_\la^{-1})-\ell(u u_\la^{-1})}  X^u
  E_\la^{A_{2n}^{(2)\dag}}(X;q,v) = \sum_{p\in \BPQ_{A_{2n}^{(2)\dag}}(u; m_\la)}
  X^{\wt(p)} q^{\deg(\qwt_{A_{2n}^{(2)\dag}}(p))}.
\end{align}
\end{prop}

\begin{ex}\label{X:a2dag} Consider the situation of Example \ref{X:a2} except for the root system $A_{2n}^{(2)\dagger}$. The only difference is that instead of truncating Bruhat edges associated with roots
$\beta_j$ with $i_j=0$, we cut quantum edges labeled by $\beta_j$ with $i_j=0$. Instead of cutting the edge $p_{1}\to p_{14}$ we cut the edge $p_{123}\to p_{1234}$. We have
\begin{align*}
  E_{(-1,0)}^{A_{2\cdot2}^{(2)\dag}}(X;q;0) = X^{(-1,0)} + X^{(0,-1)} + X^{(0,1)}  + X^{(1,0)} + X^{(0,0)}.
\end{align*}
\end{ex}

\section{At $v\to\infty$}
In this section we study the limit where all Hecke parameters are set to $v$ and then $v$ is sent to $\infty$.

\subsection{Duality of Macdonald polynomials}
Let $w_0\in W_0$ be the long element.

The following result is stated in \cite{Che:2005} in the dual untwisted case but the proof is the same for the general case.

\begin{prop}\label{P:star}
\cite[Prop. 3.3.3]{Che:2005} 
\begin{align}\label{E:starE}
  E_{-w_0(\la)}^* = v^{-1}_{w_0 u_\la^{-1}} v_{u_\la} T_{w_0} E_\la
\end{align}
where $(X^\mu)^* = X^{-\mu}$ for all $\mu\in X$, $q^* = q^{-1}$ and $v_\alpha^* = v_\alpha^{-1}$ for all $\alpha$.
\end{prop}

Let $[X^\mu]f$ denote the coefficient of $X^\mu$ in $f$.

\begin{lem}\label{L:dual} For all weights $\la,\mu\in X$,
\begin{align}\label{E:dual}
  [X^{-w_0(\mu)}] E_{-w_0(\la)}(X;q;v_\bullet) =
  [X^\mu] E_\la(X;q;v_\bullet).
\end{align}
\end{lem}
\begin{proof} Conjugation by $w_0$ induces an automorphism of the finite Dynkin diagram $*: I_0 \to I_0$ defined by
$w_0 s_i w_0 = s_{i^*}$. This induces an automorphism of $X$ by $\la\mapsto -w_0\la$. Defining $0^*=0$ we have an affine Dynkin automorphism fixing $0$. One may check that the intertwiner $\psi_{m_\la}$ (viewed, say, as an element of $\mathrm{End}(\K\tX)$) is sent to $\psi_{m_{-w_0(\la)}}$ under this automorphism. The Lemma follows.
\end{proof}

\begin{prop}\label{P:alcovepathdual}
For $\la\in X$, there is a bijection $*:\BP(w_0; m_\la)\to \BP(\id; m_\la)$ denoted $p_J\mapsto p_J^*$ such that $z_k^* = w_0 z_k$ and $J^{*+}=J^-$ and $J^{*-}=J^+$ where $J^{*\pm}$ indicates sets of positive and negative folds for $p^*_J$. Moreover
$\ep(p^*_J)=w_0(\ep(p_J))$, $\wt(p^*_J)=w_0(\wt(p_J))$, and $\dir{p^*_J}=w_0(\dir{p_J})$.
\end{prop}
\begin{proof} This follows from the definitions.
The negation of signs holds since
$w_0(\Z\delta^Y\pm R_+(Y)) = \Z\delta^Y \mp R_+(Y)$.
\end{proof}

\subsection{Specialization at $v=\infty$}
Let $(X,Y)$ be a double affine datum, $\la\in X$,
$m_\la= \pi^Y s_{i_1}\dotsm s_{i_\ell}$ reduced, $\wa=(i_1,\dotsc,i_\ell)$.
Let $\BPQR(u; m_\la)$ be the set of alcove paths $p_J\in \BP(u; m_\la)$ which project to the \textit{reverse} of a path in $\QB(\tX)$: letting $z_0,\dotsc,z_r$ as before, we have the path in $\QB(\tX)$
\begin{align}
  \dir{z_0} \overset{-\bbeta_{j_1}}{\leftarrow} \dir{z_1}
  \overset{-\bbeta_{j_2}}{\leftarrow} \dotsm
  \overset{-\bbeta_{j_r}}{\leftarrow} \dir{z_r}.
\end{align}
We define 
\begin{align}
\label{E:dqwtdef}
\qwt^*(p_J) = \sum_{j\in J^+} \beta_j.
\end{align}
The following result says that the $v=\infty$ specialization is like the $v=0$ specialization, except that the alcove path must project to the reverse of a path in the quantum Bruhat graph. The edges in the quantum Bruhat graph that contribute to the power of $q$ are still the quantum edges, but since these edges and those in the alcove path go in opposite directions, positive folds contribute to the power of $q$ rather than negative.

\begin{prop}\label{P:RYinfinity}
Let $\la\in X$, $m_\la=\pi^Y s_{i_1}\dotsm s_{i_\ell}$ for $\wa=(i_1,\dotsc,i_\ell)$ reduced. Then
\begin{align}\label{E:RYinfinity}
  E_\la(X;q^{-1};\infty) &= \sum_{p\in \BPQR(\id; m_\la)} X^{\wt(p)} q^{\deg(\qwt^*(p))}.
\end{align}
\end{prop}
\begin{proof} First set all Hecke parameters equal to $v$.
We have
\begin{align*}
  [X^\mu] E_\la(X;q^{-1};v^{-1})
  &= [X^{-w_0(\mu)}] E_{-w_0(\la)}(X;q^{-1};v^{-1}) \\
  &= [X^{w_0(\mu)}] E_{-w_0(\la)}^* \\
  &= [X^{w_0(\mu)}] v^{\ell(w_0)-2\ell(u_\la^{-1})} T_{w_0} E_\la.
\end{align*}
by Lemma \ref{L:dual} and Proposition \ref{P:star}.
Now send $v\to0$. We have
\begin{align*}
  [X^\mu] E_\la(X;q^{-1};\infty) 
  &= [X^{w_0(\mu)}] \sum_{p\in \BPQ(w_0; m_\la)} X^{\wt(p)} q^{\deg(\qwt(p))}
\end{align*}
by Corollary \ref{C:quantum} with $u=w_0$ and Remark \ref{R:W0}.

By \cite{LNSSS} the long element $w_0\in W_0$ induces an arrow-reversing involution on the quantum Bruhat graph that preserves the edge labels and sends Bruhat arrows to Bruhat and quantum arrows to quantum. It follows that the bijection $*$ of Proposition \ref{P:alcovepathdual} $p\mapsto p^*$ restricts to a bijection $\BPQ(w_0; m_\la)\to \BPQR(\id; m_\la)$ such that $\wt(p^*)=w_0(\wt(p))$ and $\qwt^*(p^*) = \qwt(p)$, the latter holding because folds changed signs. The Proposition follows.
\end{proof}

In particular, \eqref{E:RYinfinity} shows that the coefficients of $E_\la(X;q^{-1};\infty)$ belong to $\Z_{\geq 0}[q]$, confirming a conjecture from \cite{CO} (and extending it to arbitrary affine type).

\begin{rem}
For any $\la\in X_-$ and $\mu\in W_0(\la)$, the coefficient $[X^\mu]E_\la(X;q^{-1};\infty)$ is a nonnegative power of $q$.
This is \cite[Corollary 2.6($i$)]{CO} in the dual untwisted case and \cite[Theorem 3.1]{CF} in the general reduced case.
This fact also follows directly from the specialization at $q=1$ of \cite[(5.7.8))]{Mac:2003}, the formula which expresses $P_{w_0(\la)}$ as a sum of $E_\mu$ for $\mu\in W_0(\la)$.
Returning to \ref{E:RYinfinity}, we deduce that for $\la\in X_-$ and $\mu\in W_0(\la)$, there exists a unique element $p\in\BPQR(\id;m_\la)$ satisfying $\wt(p)=\mu$.

We believe that a similar statement holds for arbitrary $\la\in X$
and $\mu\in W_0(\la)$, namely that $[X^\mu]E_\la(X;q^{-1};\infty)$ is either $0$ or a nonnegative power of $q$, and the latter happens only if $u_\mu\geq u_\la$ in Bruhat order on $W_0$ (but not if and only if). This should follow from the shellability of the quantum Bruhat graph \cite{BFP}.
\end{rem}

\begin{rem}
For any $\la\in X_+$, one has $P_\la(X;q;v) = P_\la(X;q^{-1};v^{-1})$ \cite[(5.3.2)]{Mac:2003}. Hence the specializations of $P_\la$ at $v=0$ and $v=\infty$ (resp. $q=0$ and $q=\infty$) are identical, up to the substitution $q\mapsto q^{-1}$ (resp. $v\mapsto v^{-1}$).
\end{rem}

\comment{
\begin{align}\label{E:Psum}
  P_\la =
  \sum_{\mu\in W_0\la} \prod_{\al\in \Inv(u_\mu)} \frac{v_\al^2-\chi_{-\al^Y,\,\mu}}{1-\chi_{-\al^Y,\,\mu}}E_\mu
\end{align}
where $\al^Y\in (R_0)_+(Y)$ satisfies $s_{\al^Y}=s_\al$.
}

\section{At $q^\pm \to 0$}

\subsection{At $q=0$}
We set all Hecke parameters to $v$ and set $q$ to $0$ in $E_\la(X;q;v)$. The following is due to Schwer \cite{Sc} for $P_\la(X;0;v)$. Let $\BP^+(u; m_\la)$ be the set of alcove paths with all folds positive. Define $\BP^-$ similarly but with all folds negative.

\begin{thm} \label{T:Eq0}
For $u\in W_0$ and $\la\in X$, let
$m_\la = \pi^Y s_{i_1}\dotsm s_{i_\ell}$ with $\wa=(i_1,\dotsc,i_\ell)$ reduced. Then
\begin{align}
  X^u E_\la(X;0;v) = v^{-\ell(u_\la^{-1})} \sum_{p_J\in \BP^+(u; m_\la)} X^{\wt(p_J)} v^{\ell(\dir{p_J})} (v^{-1}-v)^{|J|}.
\end{align}
\end{thm}
\begin{proof} This follows from Theorem \ref{T:TMac} using Remarks \ref{R:RYreduced} and \ref{R:tlainversions}.
\end{proof}


\subsection{At $q=\infty$}
Suppose all Hecke parameters have been set equal to $v$
and $q$ is sent to $\infty$.

\begin{thm} For all $\la\in X$,
\begin{align}
\label{E:qinf}
E_\la(X;\infty;v^{-1}) 
&=v^{\ell(w_0)-2\ell(u_\la)} \sum_{p_J\in \BP^-(\id; m_\la)} X^{\wt(p_J)} v^{\ell(w_0(\dir{p_J}))} (v^{-1}-v)^{|J|}
\end{align}
where $\BP^{-}(\id;m_\la)$ is the subset of alcove paths with all folds negative.  
\end{thm}
\begin{proof} The proof of Theorem \ref{T:Eq0} shows that
\begin{align*}
  [X^\mu] E_\la(X;q^{-1};v^{-1})
  &= [X^{w_0(\mu)}] v^{\ell(w_0)-2\ell(u_\la^{-1})} T_{w_0} E_\la.
\end{align*}
Letting $q\to0$ while using Remark \ref{R:tlainversions} and Theorem \ref{T:Eq0} with $u=w_0$ we obtain
\begin{align*}
 & \quad\,\,[X^\mu] E_\la(X;\infty;v^{-1}) \\
  &= [X^{w_0(\mu)}] v^{\ell(w_0)-2\ell(u_\la^{-1})} 
  \sum_{p_J\in \BP^+(w_0;m_\la)} X^{\wt(p_J)} v^{\ell(\dir{p_J})} (v^{-1}-v)^{|J|}.
\end{align*}
The map $p_J\mapsto p_J^*$ in the proof of Lemma \ref{L:dual} restricts to a bijection
$\BP^+(w_0; m_\la) \to \BP^-( \id;m_\la)$ since the signs of the folds are reversed. The Theorem follows.
\end{proof}


\begin{thebibliography}{RY}

\bibitem{BBL}
B. Brubaker, D. Bump, and A. Licata.
Whittaker Functions and Demazure Operators.
preprint, arXiv:1111.4230.

\bibitem{BFP} 
F. Brenti, S. Fomin, and A. Postnikov.
Mixed Bruhat operators and Yang-Baxter equations for Weyl groups.
\textit{Internat. Math. Res. Notices} 1999, no. 8, 419--441. 

\bibitem{ChaFK}
V. Chari, G. Fourier, T. Khandai.
A categorical approach to Weyl modules.
\textit{Transform. Groups} 15 (2010), no. 3, 517--549.

\bibitem{Che:1992} I. Cherednik.
Double affine Hecke algebras, Knizhnik-Zamolodchikov equations, and Macdonald's operators.
\textit{Internat. Math. Res. Notices} (1992), no. 9, 171--180. 

\bibitem{Che:1995} I. Cherednik.
Nonsymmetric Macdonald polynomials.
\textit{Int. Math. Res. Notices} 1995 (10) (1995) 483--515.

\bibitem{Che:1997} I. Cherednik.
Intertwining operators of double affine Hecke algebras.
\textit{Selecta Math. (N.S.)} 3 (1997), no. 4, 459--495. 

\bibitem{Che:2005} I. Cherednik.
Double affine Hecke algebras. 
London Mathematical Society Lecture Note Series, 319.
Cambridge University Press, Cambridge, 2005.

\bibitem{CF} I. Cherednik and E. Feigin.
Extremal part of the PBW-filtration and E-polynomials.
preprint, arXiv:1306.3146.

\bibitem{CO} I. Cherednik and D. Orr.
Nonsymmetric difference Whittaker functions.
preprint, arXiv:1302.4094.

\bibitem{FL} G. Fourier and P. Littelmann.
Weyl modules, Demazure modules, KR-modules, crystals, fusion products and limit constructions. 
\textit{Adv. Math.} 211 (2007), no. 2, 566--593.

\bibitem{GL} S. Gaussent and P. Littelmann.
LS galleries, the path model, and MV cycles.
\textit{Duke Math. J.} 127 (2005), no. 1, 35--88.

\bibitem{Go} U. G\"ortz/
Alcove walks and nearby cycles on affine flag manifolds. 
\textit{J. Algebraic Combin.} 26 (2007), no. 4, 415--430.

\bibitem{HHL} J. Haglund, M. Haiman, and N. Loehr.
A combinatorial formula for nonsymmetric Macdonald polynomials.
\textit{Amer. J. Math.} 130 (2008), no. 2, 359--383. 

\bibitem{Hai} M. Haiman, 
Cherednik algebras, Macdonald polynomials and combinatorics. \textit{International Congress of Mathematicians.} Vol. III, 843--872, Eur. Math. Soc., Z\"urich, 2006.

\bibitem{I:2003} B. Ion.
Nonsymmetric Macdonald polynomials and Demazure characters.
\textit{Duke Math. J.} 116 (2003), no. 2, 299--318.

\bibitem{I:2008} B. Ion.
Standard bases for affine parabolic modules and nonsymmetric Macdonald polynomials.
\textit{Journal of Algebra} 319 (2008) 3480--3517.

\bibitem{Kac} V. Kac. 
Infinite dimensional Lie algebras.
Third edition. Cambridge University Press, Cambridge, 1990.

\bibitem{Kn} F. Knop.
Integrality of two variable Kostka functions.
\textit{J. Reine Angew. Math.} 482 (1997), 177--189.

\bibitem{LS} T. Lam and M. Shimozono.
Quantum cohomology of $G/P$ and homology of affine Grassmannian.
\textit{Acta Math.} 204 (2010), 49--90.

\bibitem{Len} C. Lenart.
From Macdonald polynomials to a charge statistic beyond type A.
\textit{J. Combin. Theory Ser. A} 119 (2012), no. 3, 683--712. 

\bibitem{LenSc} C. Lenart and A. Schilling,
Crystal energy functions via the charge in types A and C.
\textit{Math. Z.} 273 (2013), no. 1-2, 401--426. 


\bibitem{LNSSS} C. Lenart, S. Naito, D. Sagaki, A. Schilling and M. Shimozono. 
A uniform model for Kirillov-Reshetikhin crystals I: Lifting the parabolic quantum Bruhat graph.
preprint, arXiv:1211.2042.

\bibitem{LNSSS2} C. Lenart, S. Naito, D. Sagaki, A. Schilling and M. Shimozono, in preparation.

\bibitem{Mac:1996} 
I.G. Macdonald.
Affine Hecke algebras and orthogonal polynomials, in: S\'eminaire Bourbaki, vol. 1994/1995,
Ast\'erisque 237 (1996) 189--207, Exp. No. 797, 4.

\bibitem{Mac:2003}
I.G. Macdonald.
Affine Hecke algebras and orthogonal polynomials.
Cambridge Tracts in Mathematics, 157.
Cambridge University Press, Cambridge, 2003.

\bibitem{O} E. Opdam.
Harmonic analysis for certain representations of graded Hecke algebras.
\textit{Acta Math.} 175 (1995) 75--121.


\bibitem{RY} A. Ram and M. Yip.
A combinatorial formula for Macdonald polynomials.
\textit{Adv. Math.} 226 (2011), no. 1, 309--331. 

\bibitem{Sage-Combinat}
The {S}age-{C}ombinat community, {\emph{{S}age-{C}ombinat}: enhancing Sage as a toolbox for computer exploration in algebraic combinatorics}, {{\tt http://combinat.sagemath.org}}, 2008.

\bibitem{Sah}
S. Sahi.
Nonsymmetric Koornwinder polynomials and duality.
\textit{Ann. of Math.} (2) 150 (1) (1999) 267--282.

\bibitem{San} Y. Sanderson.
On the connection between Macdonald polynomials and Demazure characters. 
\textit{J. Algebraic Combin.} 11 (2000), no. 3, 269--275.

\bibitem{Sc} C. Schwer.
Galleries, Hall-Littlewood polynomials, and structure constants of the spherical Hecke algebra.
\textit{Int. Math. Res. Not.} 2006, Art. ID 75395, 31 pp.

\bibitem{St} J. Stokman.
Macdonald-Koornwinder polynomials.
preprint, arXiv:1111.6112.

\end{thebibliography}
\end{document}